\documentclass[oneside,11pt]{article}

\usepackage{graphicx, subfigure}
\usepackage[top=1in,bottom=1in,left=1in,right=1in]{geometry}
\usepackage[sort&compress]{natbib} 
\usepackage{amsfonts, amsmath, amssymb, amsthm, constants, bbm}
\usepackage{MnSymbol}
\usepackage{mathtools}
\usepackage{enumitem}
\usepackage{caption}
\usepackage{algpseudocode}
\usepackage{algorithm}

\usepackage{cellspace}
\usepackage{makecell}
\setcellgapes{4pt}

\usepackage[dvipsnames]{xcolor}
\definecolor{darkred}{RGB}{100,0,0}
\definecolor{darkgreen}{RGB}{0,100,0}
\definecolor{darkblue}{RGB}{0,0,150}
\definecolor{citecol}{RGB}{30,80,150}
\definecolor{tabcol}{RGB}{200,230,255}

\usepackage{colortbl}

\usepackage{hyperref}
\hypersetup{colorlinks=true, linkcolor=darkred, citecolor=citecol, urlcolor=darkblue}
\usepackage{url}

\newtheorem{thm}{Theorem}[section]
\newtheorem{prp}[thm]{Proposition}
\newtheorem{lem}[thm]{Lemma}

\newtheorem{dfn}[thm]{Definition}
\newtheorem{ass}[thm]{Assumption}

\theoremstyle{remark}
\newtheorem{rem}{Remark}[section]

\def\beq{\begin{equation}} 
\def\eeq{\end{equation}}
\def\beqn{\begin{eqnarray*}}
\def\eeqn{\end{eqnarray*}}
\def\Bal{\begin{align}}
\def\Eal{\end{align}}
\def\Bitem{\begin{itemize}\setlength{\itemsep}{.2in}}
\def\bitem{\begin{itemize}\setlength{\itemsep}{.05in}}
\def\eitem{\end{itemize}}
\def\blatin{\begin{enumerate}\setlength{\itemsep}{.05in}\renewcommand{\labelenumi}{\roman{enumi}.}}
\def\elatin{\end{enumerate}}
\def\Benum{\begin{enumerate}\setlength{\itemsep}{.2in}}
\def\benum{\begin{enumerate}\setlength{\itemsep}{.05in}}
\def\eenum{\end{enumerate}}
\def\bmult{\begin{multline*}}
\def\emult{\end{multline*}}
\def\bcenter{\begin{center}}
\def\ecenter{\end{center}}
\def\bframe{\begin{frame}}
\def\eframe{\end{frame}}

\newcommand{\thmref}[1]{Theorem~\ref{thm:#1}}
\newcommand{\prpref}[1]{Proposition~\ref{prp:#1}}

\newcommand{\lemref}[1]{Lemma~\ref{lem:#1}}
\newcommand{\secref}[1]{Section~\ref{sec:#1}}
\newcommand{\appref}[1]{Appendix~\ref{app:#1}}
\newcommand{\figref}[1]{Figure~\ref{fig:#1}}

\newcommand{\algoref}[1]{Algorithm~\ref{algo:#1}}
\newcommand{\remref}[1]{Remark~\ref{rem:#1}}

\newcommand{\assref}[1]{Assumption~\ref{ass:#1}}

\def\cA{\mathcal{A}}
\def\cB{\mathcal{B}}
\def\cC{\mathcal{C}}

\def\cF{\mathcal{F}}
\def\cG{\mathcal{G}}
\def\cH{\mathcal{H}}

\def\cR{\mathcal{R}}

\def\bn{\mathbf{n}}

\def\bbE{\mathbb{E}}

\def\bbH{\mathbb{H}}

\def\bbN{\mathbb{N}}

\def\bbP{\mathbb{P}}

\def\bbR{\mathbb{R}}

\def\bbT{\mathbb{T}}

\def\bbZ{\mathbb{Z}}
\def\ind{\mathbbm{1}}

\newcommand{\Var}{\operatorname{Var}}

\def\wh{\widehat}

\renewcommand{\>}{\rangle}
\newcommand{\inner}[2]{\langle #1, #2 \rangle}

\newcommand{\floor}[1]{\lfloor #1 \rfloor}
\newcommand{\ceil}[1]{\lceil #1 \rceil}

\def\1{\mathbbm{1}}

\def\({\left(}
\def\){\right)}

\DeclareMathOperator{\tv}{TV}

\DeclareMathOperator{\op}{op} 

\DeclareMathOperator{\Leb}{Leb}

\DeclareMathOperator{\Id}{Id}

\DeclareMathOperator{\II}{I\!I}

\DeclareMathOperator{\pr}{pr}

\DeclareMathOperator{\supp}{supp}
\DeclareMathOperator{\vol}{vol}
\DeclareMathOperator{\inj}{inj}

\newcommand{\ve}{\varepsilon}
\newcommand{\cdt}{\mathcal{C}_{d} (\tau)}
\newcommand{\cdatl}{\mathcal{C}_{d,\alpha} (\tau,L)}

\newcommand{\hfhx}{\widehat f_{h}(x)}

\newcommand{\sdab}{\Sigma^d_{\alpha,\beta}}
\DeclareMathOperator{\od}{1D}
\DeclareMathOperator{\ad}{adapt}

\newcommand{\clement}[1]{\textcolor{black}{#1}}
\newcommand{\marc}[1]{\textcolor{black}{#1}}

\begin{document}

\thispagestyle{empty}

\title{Density estimation on an unknown submanifold}
\author{Cl\'ement Berenfeld\footnote{Cl\'ement Berenfeld, Universit\'e Paris-Dauphine \& PSL, CNRS, CEREMADE, 75016 Paris, France. email: berenfeld@ceremade.dauphine.fr}\;\; and Marc Hoffmann\footnote{Marc Hoffmann, Universit\'e Paris-Dauphine \& PSL, CNRS, CEREMADE, 75016 Paris, France. email: hoffmann@ceremade.dauphine.fr}
}
\date{\today}
\maketitle

\begin{abstract}
We investigate density estimation from a $n$-sample in the Euclidean space $\mathbb R^D$, when the data are supported by an unknown submanifold $M$ of possibly unknown dimension $d < D$, under a {\it reach} condition. We investigate several nonparametric kernel methods, with data-driven bandwidths that incorporate some learning of the geometry via a local dimension estimator. When $f$ has H\"older smoothness $\beta$ and $M$ has regularity $\alpha$, our estimator achieves the rate $n^{-\alpha \wedge \beta/(2\alpha \wedge \beta+d)}$ for a pointwise loss. The rate does not depend on the ambient dimension $D$ and we establish that our procedure is asymptotically minimax for $\alpha \geq \beta$. 
Following Lepski's principle, a bandwidth selection rule is shown to achieve smoothness adaptation. We also investigate the case $\alpha \leq \beta$: by estimating in some sense the underlying geometry of $M$, we establish in dimension $d=1$ that the minimax rate is $n^{-\beta/(2\beta+1)}$ proving in particular that it does not depend on the regularity of $M$.
Finally, a numerical implementation is conducted on some case studies in order to confirm the practical feasibility of our estimators.
\end{abstract}

\vspace{3mm}



\section{Introduction}

\subsection{Motivation} \label{sec:motiv}
Suppose we observe an $n$-sample
$(X_1,\ldots, X_n)$ of size $n$ distributed on an Euclidean space $\mathbb R^D$ according to some density function $f$. We wish to recover $f$ at some point arbitrary point $x \in \mathbb R^D$ nonparametrically. If the smoothness of $f$ at $x$ measured in a strong sense is of order $\beta$ -- for instance by a H\" older condition or with a prescribed number of derivatives -- then the optimal (minimax) rate for recovering $f(x)$ is of order $n^{-\beta/(2\beta+D)}$ and is achieved by kernel or projection methods, see {\it e.g.} the classical textbooks \cite{S86, DG85} or \citep[Sec. 1.2-1.3]{Tsy08}. Extension to data-driven bandwidths \citep{B84, C91} offers the possibly to adapt to unknown smoothness, see \citep{GL08, GL11, GL14} for a modern mathematical formulation. More generally, recommended reference on adaptive estimation is the textbook by \cite{gine2016mathematical}. In many situations however, the dimension $D$ of the ambient space is large, hitherto disqualifying such methods for pratical applications. Opposite to the curse of dimensionality, a broad guiding principle in practice is that the observations $(X_1,\ldots, X_n)$ actually live on smaller dimensional structures and that the effective dimension of the problem is smaller if one can take advantage of the geometry of the data \citep{FMN16}. This classical paradigm probably goes back to a conjecture of \citep{S82} that paved the way to the study of the celebrated single-index model in nonparametric regression, where a structural assumption is put in the form $f(x) = g(\langle \vartheta, x\rangle)$, where $\langle \cdot, \cdot \rangle$ is the scalar product on $\mathbb R^D$, for some unknown univariate function $g:\mathbb R \rightarrow \mathbb R$ and direction $\vartheta \in \mathbb R^D$. Under appropriate assumptions, the minimax rate of convergence for recovering $f(x)$ with smoothness $\beta$ drops to $n^{-\beta/(2\beta+1)}$ and does not depend on the ambient dimension $D$, see {\it e.g.} \citep{GL07, LS14} and the references therein. Also, in the search for significant variables, one postulates that $f$ only depends on $d<D$ coordinates, leading to the structural assumption $f(x_1,\ldots, x_D) = F(x_{i_1}, \ldots x_{i_d})$ for some unknown function $F: \mathbb R^d \rightarrow \mathbb R$ and $\{i_1,\ldots, i_d\} \subset \{1,\ldots, D\}$. In an analogous setting, the minimax rate  of convergence becomes $n^{-\beta/(2\beta+d)}$ and this is also of a smaller order of magnitude than $n^{-\beta/(2\beta+D)}$, see \citep{HL02} in the white noise model.\\

The next logical step is to assume that the data $(X_1,\ldots, X_n)$ live on a $d$-dimensional submanifold $M$ of the ambient space $\mathbb R^D$. When the manifold is known prior to the experiment, nonparametric density estimation dates back to \citep{DG85} when $M$ is the circle, and on a homogeneous Riemannian manifold by \citep{Hen90}, see also \cite{Pel05}. Several results are known for specific geometric structures like the sphere or the torus involved in many applied situations: inverse problems for cosmological data \citep{KK02, KK09, KPP11}, in geology \citep{HWC87} or flow calculation in fluid mechanics \citep{ES00}. For genuine compact homogeneous Riemannian manifolds, a general setting for smoothness adaptive density estimation and inference has recently been considered by \cite{KNP12}, or even in more abstract metric spaces in \cite{Cal18}. See also \cite{baldi2009adaptive, castillo2014thomas} and the references therein. A common strategy adapts conventional nonparametric tools like projection or kernel methods to the underlying geometry, via the spectral analysis of the Beltrami-Laplace operator on $M$. Under appropriate assumptions, this leads to exact or approximate eigenbases (spherical harmonics for the sphere, needlets and so on) or properly modified kernel methods, according to the Riemannian metric on $M$.
\\

If the submanifold $M$ itself is unknown, getting closer in spirit to a dimension reduction approach, the situation becomes drastically different: $M$ hence its geometry is unknown, and considered as a nuisance parameter. In order to recover the density $f$ at a given point $x \in \mathbb R^D$ of the ambient space, one has to understand the minimal geometry of $M$ that must be learned from the data and how this geometry affects the optimal reconstruction of $f$. This is the topic of the paper.  \\

\marc{We consider in the paper a seemingly unusual framework where the support of a distribution is unknown while the aim is to recover the density at a point $x \in \mathbb R^D$ which is known to be on the support. As mentioned above}, this actually covers at least two situations:
\bitem
\item The data are high-dimensional and it is reasonable to believe that they actually lie on a smaller dimensional subset of the ambient space $\mathbb R^D$, which can be assumed to be a submanifold. In that case, $x$ can be seen as an observation $X_0$ from our dataset, and the analysis can be (implicitly) performed conditional on $X_0$;
\item The data naturally lie on a submanifold, like a spheroid for geological application, or a cell membrane in microbiology (see for instance \cite{KPS14} who describe a technique that yields such a point cloud). In this case, $x$ can be seen as an observation $X_0$ like above, but there is also the situation where the statistician can know whether or not a given point $x$ is within the support (for instance a point on a cell membrane, or a geographical location on the Earth surface) without knowing the geometric feature of the latter and without needing to estimate them.
\eitem

\subsection{Main results}

We construct a class of compact smooth submanifolds of dimension $d$ of the Euclidean space $\mathbb R^D$, without boundaries, that constitute generic models for the unknown support of the target density $f$ that we wish to reconstruct. We further need a {\it reach condition}, a somehow unavoidable notion in manifold reconstruction that goes back to \cite{Fed59}: it is a geometric invariant that quantifies both local curvature conditions and how tightly the submanifold folds on itself. It is related to the scale at which the sampling rate $n$ can effectively recover the geometry of the submanifold, see Section \ref{sec:reach} below.
We consider regular manifolds $M$ with reach bounded below that satisfy the following property: $M$ admits a local parametrization at every point $x \in M$ by its tangent space $T_x M$, and this parametrization is sufficiently regular. A natural candidate is given by the exponential map $\exp_x: T_xM \rightarrow M \subset \bbR^D$. More specifically, for some regularity parameter \marc{$\alpha \geq  0$}, we require a certain uniform bound for the \marc{$\alpha+1$}-fold differential of the exponential map to hold, quantifying in some sense the regularity of the parametrization in a minimax spirit. \marc{Our approach is close to that of \citet[Def. 1]{AL17} that consider arbitrary parametrizations among those close to the inverse of the projection onto tangent spaces}. Given a density function $f:M \rightarrow [0,\infty)$ with respect to the volume measure on $M$, we have a natural extension of smoothness spaces on $M$ by requiring that $f \circ \exp_x: T_xM \rightarrow \mathbb R$ is a smooth map in any reasonable sense, see for instance \cite{Tri87} for the characterisation of function spaces on a Riemannian manifold.\\

Our main result is that in order to reconstruct $f(x)$ efficiently at a point $x \in \mathbb R^D$ \marc{when $f$ has smoothness $\beta$ and lives on an unknown submanifold of smoothness $\alpha$ and unknown dimension $d<D$}, it is sufficient to consider estimators of the form
\begin{equation} \label{eq: def est}
\widehat f_h(x) = \frac{1}{nh^{\widehat d(x)}} \sum_{i =1}^n K\(\frac{x - X_i}{h}\),~~ x \in \mathbb R^D,
\end{equation}
where $K: \mathbb R^D \rightarrow \mathbb R$ is a certain kernel and $\widehat d(x) = \widehat d(x, X_1,\ldots, X_n)$ is an estimator of the local dimension of the support of $f$ in the vicinity of $x$ based on a scaling estimator as introduced in \cite{FSA07}. We prove in \thmref{upper} that following a classical bias-variance trade-off for the bandwidth $h$, the rate $n^{-\alpha \wedge \beta/(2\alpha\wedge \beta +d)}$ is achievable for pointwise and global loss when the dimension of $M$ is $d$, irrespectively of the ambient dimension $D$. In particular, it is noteworthy that in terms of manifold learning, only the dimension of $M$ needs to be estimated. When $\alpha \geq \beta$, we also have a lower bound (\thmref{lower}) showing that our result is minimax optimal. Moreover, by implementing Lepski's principle \citep{Lep92}, we are able to construct a data driven bandwidth $\widehat h = \widehat h(x,X_1\ldots, X_n)$ that achieves  in Theorem \ref{thm:adapt} the rate $n^{-\alpha \wedge \beta/(2\alpha\wedge \beta +d)}$ up to a logarithmic term --- unavoidable in the case of pointwise loss due to the Lepski-Low phenomenon \citep{L90, Lo92}. When the dimension $d$ is known, the estimator \eqref{eq: def est} has already been investigated in squared-error norm in \cite{OG09} for a fixed manifold $M$ and smoothness $\beta = 2$.\\


A remaining \marc{issue} at this stage is to understand how the regularity of $M$ can affect the minimax rates of convergence for smooth functions, {\it i.e.} when $\alpha \leq \beta$. \marc{We only have a partial answer to that question, when we restrict our attention to the one-dimensional case $d=1$.} When $M$ is known, \cite{Pel05}  studied estimators of the form
\begin{equation} \label{pelest}
\frac{1}{nh^d}\sum_{i = 1}^n \frac{1}{\vartheta_{x}(X_i)}K\(\frac{d_M(x,X_i)}{h}\),
\end{equation}
where $K: \mathbb R \rightarrow \mathbb R$ is a radial kernel, 
$d_M$ is the intrinsic Riemannian distance on $M$ and 
the correction term $\vartheta_x(X_i)$ is the volume density function on $M$ \citep[p. 154]{Bes78} that accounts for the value of the density of the volume measure at $X_i$ in normal coordinates around $x$, taking into account how the submanifold curves around $X_i$. By establishing in \lemref{dim1} that $\vartheta_x$ is constant (and identically equal to one) when $d=1$, we have another estimator by simply learning the geometry of $M$ via its intrinsic distance $d_M$ in \eqref{pelest}.  This can be done by efficiently estimating $d_M$ \marc{in dimension $d=1$}  thanks to the Isomap method as coined by \cite{TDL00}. \marc{Therefore, in the special case when the dimension $d$ of $M$ is known and equal to $1$, we are able to} construct an estimator that achieves in \thmref{onedim} the rate $n^{-\beta/(2\beta+1)}$, therefore establishing that in dimension $d=1$ at least, the regularity of the manifold $M$ does not affect the minimax rate for estimating $f$ even when $M$ is unknown. However, the volume density function $\vartheta_x$ is not constant as soon as $d \geq 2$ and obtaining a global picture in higher dimensions remains \marc{an open and presumably challenging problem}.

\subsection{Organisation of the paper}

In \secref{prelim}, we provide with all the necessary material and notation from classical geometry for the unfamiliar reader. \marc{Section \ref{sec: classical geom} together with the construction of smoothness spaces -- here H\"older spaces on a submanifold in Section \ref{sec: holder spaces}. We elaborate in particular on the reach of a subset of the Euclidean space in \secref{reach} and construct a statistical model for sampling $n$ data from a density $f$ with regularity $\beta$ living on an unknown submanifold $M$ of unknown dimension $d$ and smoothness $\alpha$ in an ambient space of dimension $D$ in \secref{model}. In this setting, we establish in Section \ref{sec:loss} that a reach condition, {\it i.e.} assuming that the reach of $M$ is bounded below, is necessary in order to reconstruct $d$. This is stated precisely in Theorem \ref{thm:reachrisk}.}\\

\marc{We give our main results in Section \ref{sec: dens est} and more specifically in Section \ref{sec: dens est}. When the dimension $d$ and the smoothness parameters $\alpha$ of the unknown manifold $M$ and the smoothness $\beta$ of $f$ are known, Theorem \ref{thm:upper} states the existence of an estimator that achieves the rate $n^{-\alpha\wedge \beta/(2\alpha \wedge \beta+d)}$  in expected pointwise loss,  and Theorem \ref{thm:lower} establishes that a minimax lower bound is $n^{-\beta/(2\beta+d)}$. Theorem \ref{thm:onedim} shows the existence of estimators in dimension $d=1$ that achieve the rate $n^{-\beta/(2\beta+1)}$, which is therefore minimax in that case. Theorem \ref{thm:adapt} states the existence of smoothness and dimension adaptive estimators, when $\alpha,\beta$ and $d$ are unknown. 
Section \ref{sec:ker} elaborates on special kernels upon which the estimators that achieve the aforementioned results are constructed, and their properties with respect to bias and variance analysis. The underlying geometry of $M$ makes the usual orthogonality to non-constant polynomials of a certain degree (the order of the kernel) irrelevant, and a specific construction must be undertaken. Section \ref{sec:uni} focuses on the case of one-dimensional submanifolds $M$ when $d=1$, where we explicitly construct a kernel estimator that achieves the minimax rate of convergence, revisiting the estimator \eqref{pelest} of \cite{Pel05}  and relying on the Isomap algorithm. In Section \ref{sec:smooth}, we implement Lepski's algorithm on the bandwidth of our kernel estimators, following \cite{LMS97}; this achieves smoothness adaptation w.r.t. $\alpha \wedge \beta$. Finally, in Section \ref{sec:dim}, we build an estimator of the dimension $d$ of $M$, following ideas of \cite{FSA07} and that enables us to obtain simultaneous adaptation w.r.t. $\alpha \wedge \beta$ and $d$ by plug-in.}\\

Finally, numerical examples are developed in \secref{simu}:  we elaborate on examples of non-isometric embeddings of the circle and the torus in dimension 1 and 2 and explore in particular rates of convergence on Monte-Carlo simulations, illustrating how effective Lepski's method can be in that context. The proof are delayed until Appendix \ref{proofs}.


\section{Manifold-supported probability distributions} \label{sec:prelim}

\subsection{Some material from \marc{classical} geometry} \label{sec: classical geom}

We recall some basic notions of geometry of submanifolds of the Euclidean space $\mathbb R^D$ for the unfamiliar reader. We borrow material from the classical textbooks \cite{GHL90} and \cite{Lee06}. We endow $\bbR^D$ with its usual Euclidean product and norm, respectively denoted by $\inner{\cdot}{\cdot}$ and $\|\cdot\|$. \clement{We denote by $B(x,r)$ the open ball of $\bbR^D$ of center $x$ and radius $r$, and, for any supspace $H \subset \bbR^D$, $B_{H}(x,r)$ the open ball of in $H$ for the induced norm (namely $B_H(x,r) = H \cap B(x,r)$).}

Classicaly, the smoothness of a submanifold is defined through the regularity of its parametrizations. Because we will need to compare quantitatively the smoothness of manifold within a large class, we will have to pick one canonical way of parametrizing them. For this reason, we consider the \emph{exponential map}; for any smooth submanifold $M \subset\bbR^{D}$ and any $x \in M$, it defines a smooth parametrization
$$
\exp_x : B_{T_x M}(0,\ve) \to M
$$
of $M$ around $x$, provided that $\ve$ is chosen small enough \citep[Cor 2.89 p.86]{GHL90}. \clement{The supremum of all such $\ve$ is called the injectivity radius at $x$ and is denoted $\inj_M(x)$.} When $M$ is a closed subset of $\bbR^D$, the exponential maps are well defined on the whole tangent spaces. This is (one side of) the Hopf-Rinow theorem \citep[Thm 6.13 p.108]{Lee06}.

Given a submanifold $M$ of dimension $d$, we will define the \emph{volume measure} of $M$, denoted by $\mu_M$, as the restriction of the $d$-dimensionnal Hausdorff measure $\cH^d$ to $M$, see \citet[Sec 2.10.2 p.171]{Fed69} for a definition. It can be shown  \citep[Ex D p.102]{EG92} that this definition coincides with the usual one of volume measure of a Riemaniann manifold, namely, if $\psi : M \rightarrow \bbR$ is a continuous fonction with support in $\exp_x(B_{T_x M}(0,\ve))$ for $\ve$ smaller than $\inj_M(x)$, we have
\beq
\mu_M(\psi) = \int_{B_{T_x M}(0,\ve)} \psi \circ \exp_x(v) \sqrt{\det g_{ij}^x (v)} dv, \nonumber
\eeq
with $g_{ij}^x(v) = \inner{d\exp_x(v)[e_i]}{d\exp_x(v)[e_j]}$ and where $(e_1,\dots,e_d)$ is an arbitrary orthonormal basis of $T_x M$. See \citet[Sec 3.H.1 and Sec 3.H.2]{GHL90} for further details on the volume measure. The volume of $M$, denoted by $\vol M$, is simply $\mu_M(\ind)$. It is finite in particular when $M$ is a compact submanifold of $\bbR^D$. 

\clement{
\subsection{H\"older spaces on submanifolds of $\bbR^D$} \label{sec: holder spaces}
Let $M$ be a smooth submanifold of $\bbR^D$. We say that a vector-valued function $f : M \to \bbR^m$ with $m \geq 1$ is $\gamma$-H\"older with $\gamma > 0$ if for all $x \in M$, the map
$$
f \circ \exp_x : B_{T_x M}(0, \rho) \to \bbR^m~~~~\text{where}~~\rho = \inj_M(x)
$$
is $\gamma$-H\"older in the usual sense, namely
\bitem
\item[(i)] $f \circ \exp_x $ is $k = \ceil{\gamma - 1}$-times differentiable;
\item[(ii)] and verifies
$$
\forall v,w \in B_{T_x M}(0,\rho),~~~~ \|d^k (f \circ \exp_x)(v) - d^k (f \circ \exp_x)(w)\|_{\op} \leq R \|v - w\|^\delta
$$
with $\delta = \gamma - k > 0$ and for some $R > 0$. 
\eitem
We will denote by $\cH^\alpha(M,\bbR^m)$ the space of all such functions, and define for $f \in \cH^\alpha(M,\bbR^m)$ the H\"older coefficient
$$
\|f\|_{\cH^\alpha} = \sup_{x \in M} \sup_{\substack{v, w \in B_{T_x M}(0,\rho)\\ \rho = \inj_M(x) }} \frac{\|d^k (f \circ \exp_x)(v) - d^k (f \circ \exp_x)(w)\|_{\op}}{\|v - w\|^\delta}.
$$
\begin{rem} The characterization of the smoothness of a function $f : M \to \bbR^m$ through the exponential maps is a classical way to define functional spaces over Riemannian manifolds, see for instance \cite{Tri87}.
\end{rem}
}

\subsection{The reach of a subset} \label{sec:reach}

One of the main concern when dealing with observations sampled from a geometrically structured probability measure is  to determine the suitable scale at which one should look at the data. Indeed, given finite-sized point cloud in $\bbR^D$, there are infinitely many submanifolds  that interpolate the point cloud, see \figref{pcloud} for an illustration. A popular notion of regularity for a subset of the Euclidean space is the \emph{reach}, introduced by \cite{Fed59}. 
\begin{figure}[ht!]
\centering
\includegraphics[scale=.08]{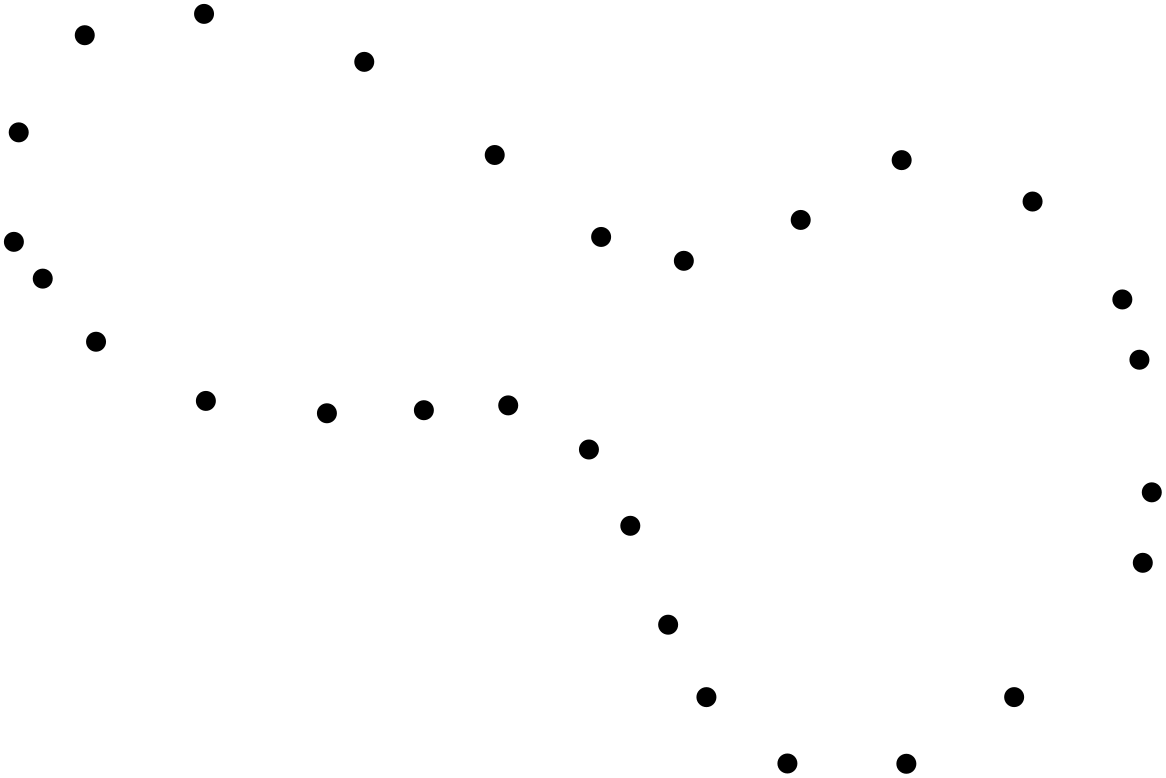} \hspace{5mm}
\includegraphics[scale=.08]{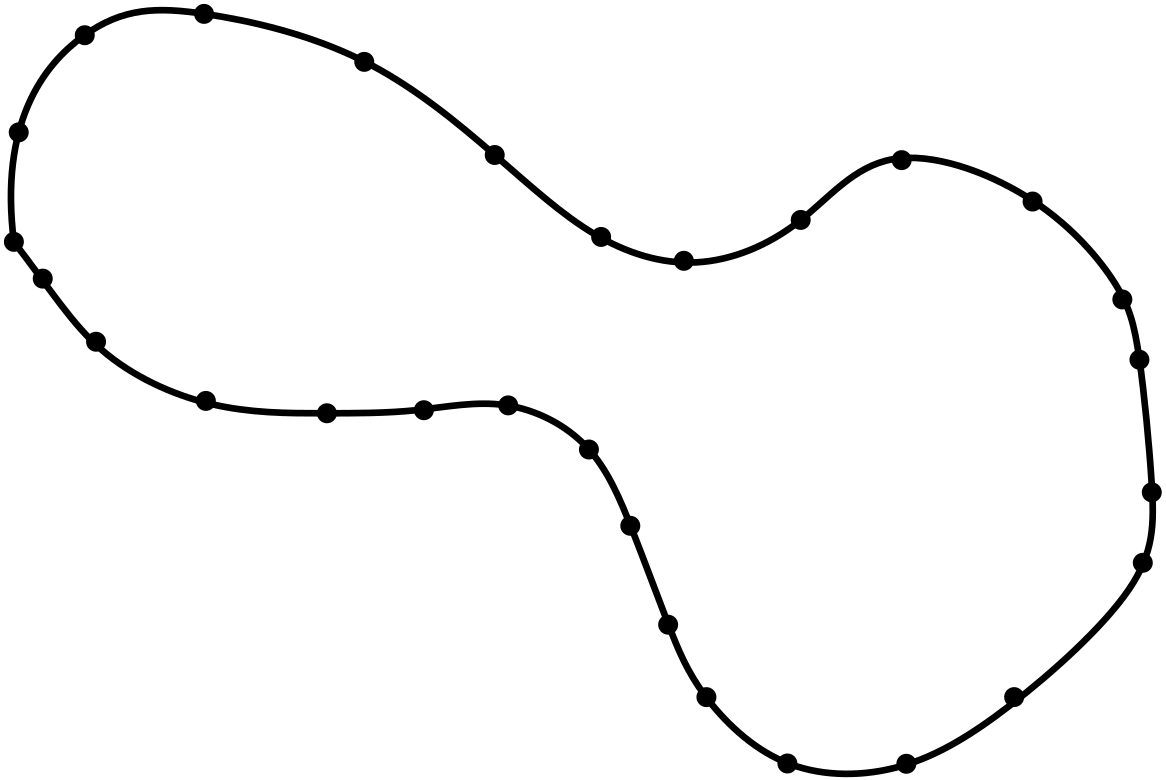} \hspace{5mm}
\includegraphics[scale=.08]{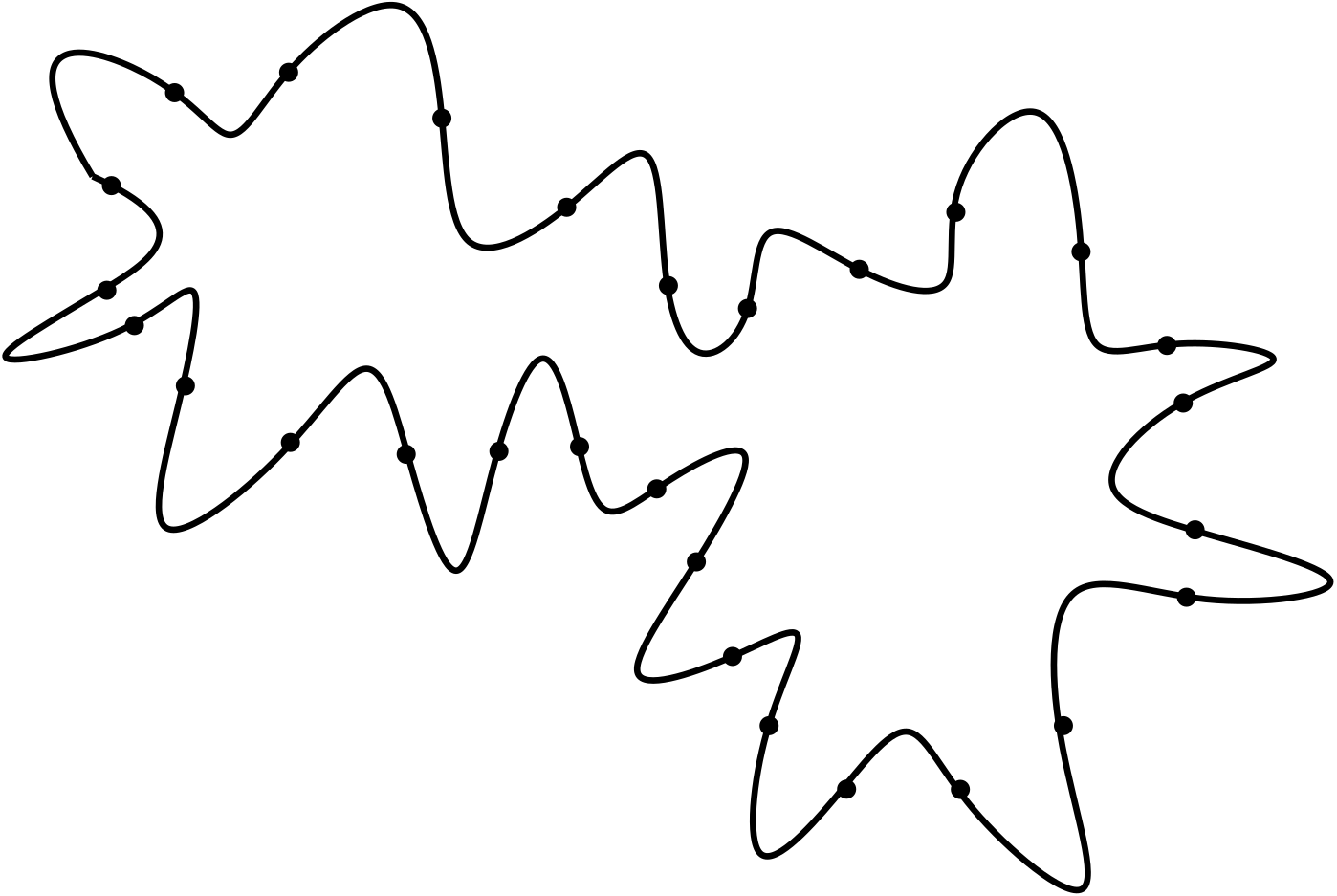}
\captionsetup{justification =  justified, margin = 1.1cm}
\caption{\small An arbitrary points cloud (Left) for $D=2$, and two smooth one-dimensional submanifolds passing through all its points (Middle, Right). A reach condition tends to discard the Right manifold as a likely candidate among all possible submanifolds the point cloud is sampled from.}
\label{fig:pcloud}
\end{figure}
\begin{dfn} Let $K$ be a compact subset of $\bbR^D$. The reach $\tau_K$ of $K$ is the supremum of all $r \geq 0$ such that the orthogonal projection $\pr_K$ on $K$ is well-defined on the $r$-neighbourhood $K^r$ of $K$, namely
\beq
\tau_K = \sup\{r \geq 0 ~|~\forall x \in \bbR^D,~d(x,K) \leq r \Rightarrow \exists ! y \in K,~ d(x,K) = \|x - y\| \}. \nonumber
\eeq
\end{dfn}
When $M$ is a compact submanifold of $\bbR^D$, the reach $\tau_M$ quantifies two geometric invariants: locally, it measures how curved the manifold is, and globally, it measures how close it is to intersect itself (the so-called bottleneck effect). See \figref{reach} for an illustration of the phenomenon. A {\it reach condition}, meaning that the reach is bounded below, is necessary in order to obtain minimax inference results in manifold learning. These include: homology inference \cite{NSW08, Bal12}, curvature \citep{AL17} and reach estimation itself \citep{AKJCMRW19} as well as manifold estimation \cite{Gal12, AL17}. 

\begin{figure}[ht!]
\centering
\includegraphics[scale=.07]{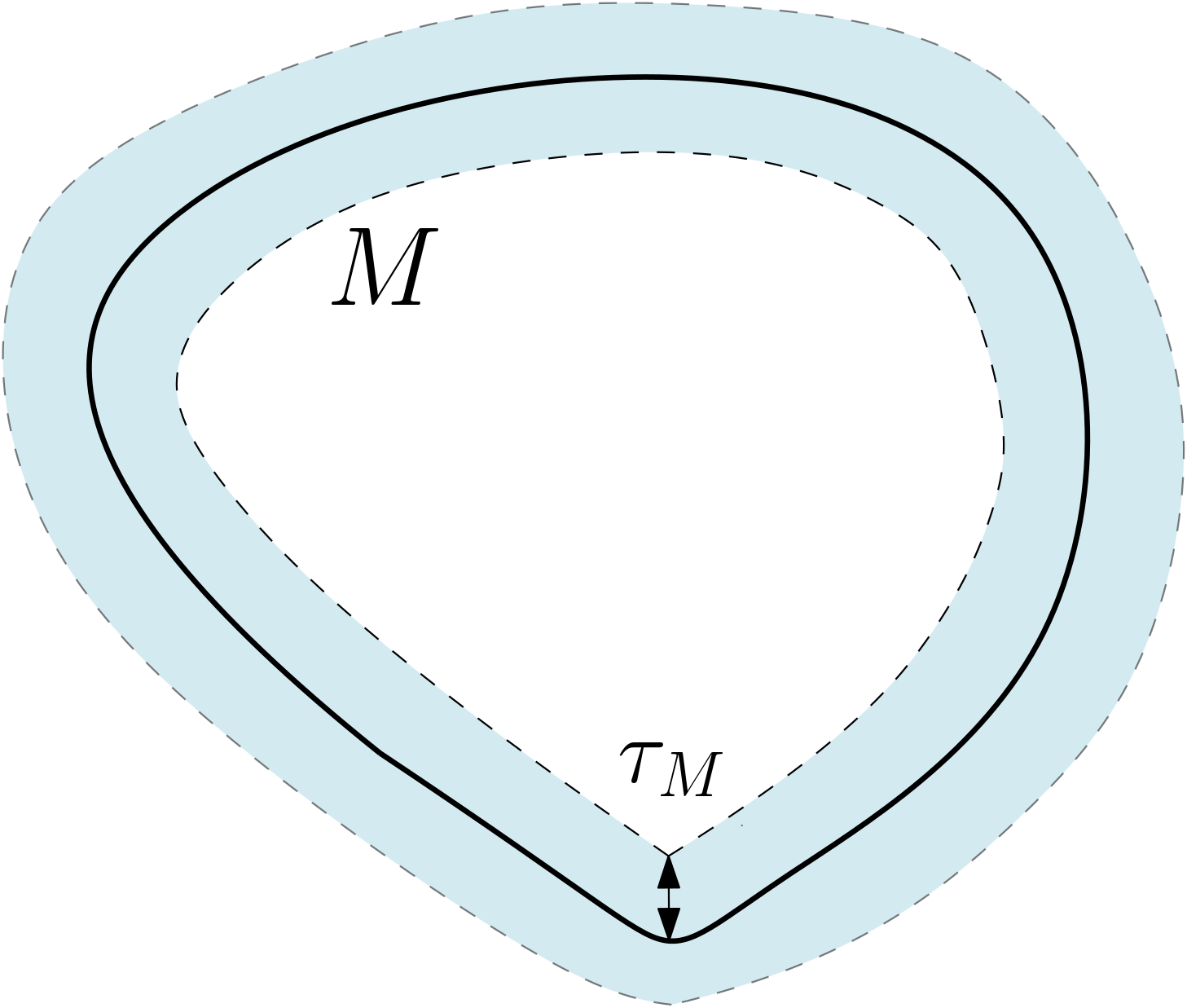} \hspace{1cm}
\includegraphics[scale=.07]{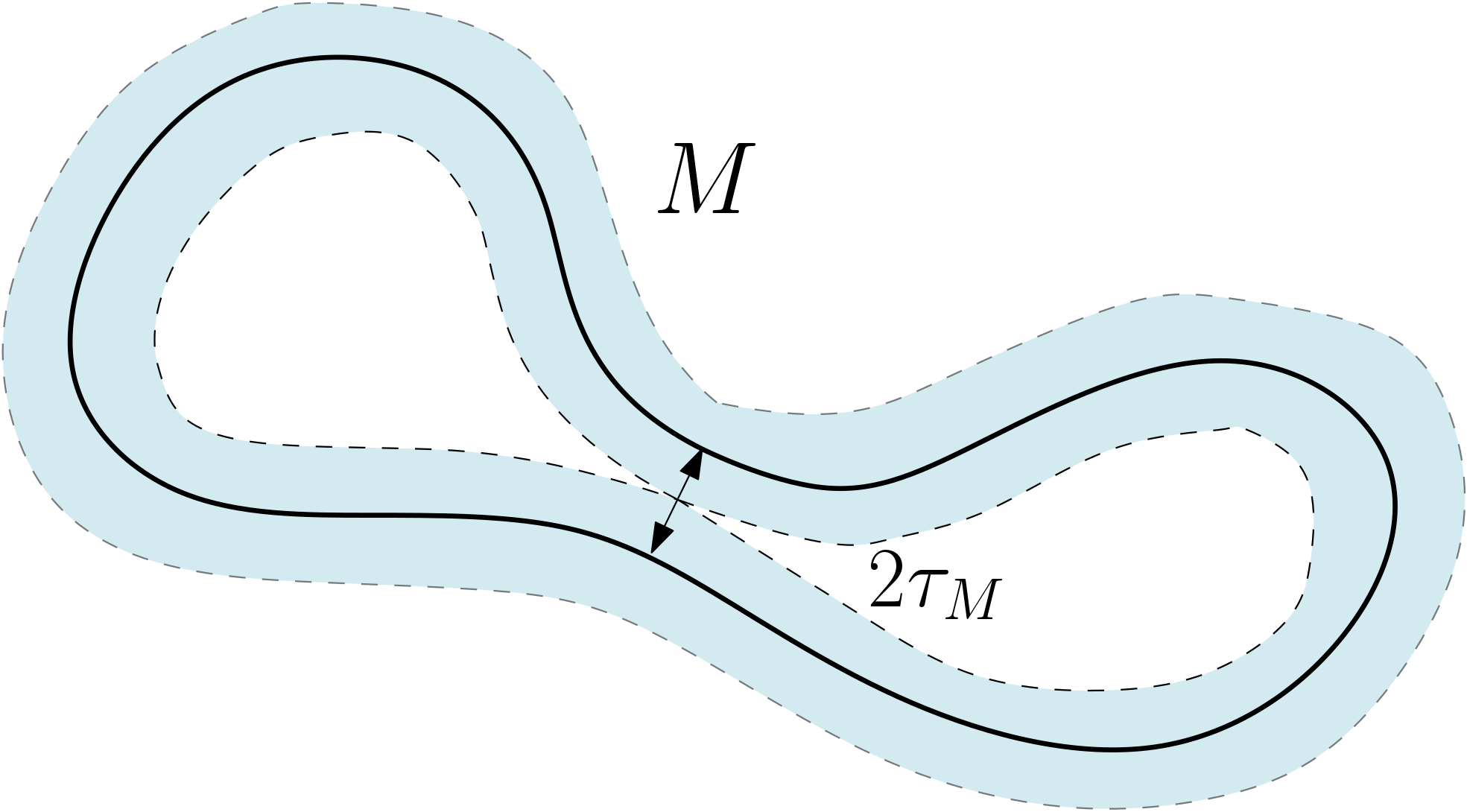}
\captionsetup{justification =  justified, margin = 1.1cm}
\caption{\small For the first manifold $M$ (Left), the value of the reach $\tau_M$ comes from its curvature. For the second one (Right), the reach is equal to $\tau_M$ because it is close to self intersecting (a bottleneck effect). The blue area represents the tubular neighbourhood over which the orthonormal projection on each manifold is well-defined.}
\label{fig:reach}
\end{figure}

\subsection{A statistical model for sampling on a unknown manifold} \label{sec:model}

\clement{
In the following, we fix a point $x \in \bbR^D$ in the ambient space. See \secref{motiv} for a discussion on such a setting.} Our statistical model is characterized by two quantities: the regularity of its support and the regularity of the density defined on this support. The support belongs to a class of submanifolds $M$,
for which we need to fix some kind of \emph{canonical} parametrization. This is what \cite{AL17} propose by asking the support $M$ to admit a local parametrization at all point $y \in M$ by $T_y M$, and that this parametrization is close to being the inverse of the projection over this tangent space. 
We \marc{follow this idea by imposing} a constraint on the exponential map. 
\begin{dfn} \label{def modele}
Let $1 \leq d < D$ be integers and \marc{$\tau > 0$}. We let $\cdt$ define the set of submanifolds $M$ of $\bbR^D$ \clement{that contains $x$} and satisfying the following properties:
\bitem
\item[(i)] (Dimension) $M$ is a smooth submanifold of dimension $d$ without boundaries;
\item[(ii)] (Compactness) $M$ is compact;
\item[(iii)] (Reach condition) We have $\tau_M \geq \tau$.
\eitem
For $\alpha \geq 0$ and  $L > 0$, we define $\cdatl$ as the set of $M \in  \cdt$ that fulfill the additional condition:
\bitem  \addtocounter{enumi}{3}
\item[(iv)] \clement{The inclusion map $\iota_M : M \hookrightarrow \bbR^D$ is $(\alpha+1)$-H\"older with $\|\iota_M\|_{\cH^{\alpha+1}} \leq L$.}
\eitem
\end{dfn}
\begin{rem} \label{rem:local}\clement{The definitions above endow our model with \emph{global} constraints, even though most of them can be stated in a \emph{local} fashion, with properties of the support holding in a neighborhood of our candidate point $x\in \bbR^D$. This meets two expectations: 
\bitem
\item[-] staying close to the existing manifold setting in statistics, like in \cite{AL17};
\item[-] allowing for further developments, like estimation in global losses, such as $L^p$-norms, Wasserstein norms, or the sup-norm \cite{WW20}.
\eitem}
\end{rem}
\clement{The reach condition \marc{$\tau_M \geq \tau > 0$} in $(iii)$ is essential in estimating consistently a density at a point in our setting, as shown in \thmref{reachrisk} in \secref{loss}.
Furthermore, a reach constraint enables the use of a number of interesting geometric results.
\begin{prp} \label{prp:injexp} Let $M$ be a compact smooth submanifold of $\bbR^D$ with $\tau_M \geq \tau$. Then the injectivity radius $\inj_M$ is everywhere greater than $\pi\tau$. 
\end{prp}
This result is a corollary of \citet[Thm 1.3]{AB06}, as explained in \citet[Lem A.1]{AL17}. Pick $M \in \cdt$. For any $z \in M$, the map $v \mapsto \exp_z(v) - x$ is bounded on $B_{T_z M}(0,\pi\tau)$ by $\pi \tau$, since for any $v \in B_{T_z M}(0,\pi\tau)$, we have $\|\exp_z(v) - x \| \leq d_M(\exp_z(v),x) =  \|v\|$, where $d_M$ is the intrinsic distance on $M$. This uniform bound along with the H\"older condition $(iv)$ allows one to obtain a uniform bound on the first derivatives of the exponential map
\begin{lem} For $M \in \cdatl$, any $z \in M$, and any $1 \leq j \leq \ceil{\alpha}$, we have
\beq
\sup_{v \in B_{T_z M}(0,\pi\tau/2)} \|d^j \exp_z(v)\|_{\op} \leq L_j, \nonumber
\eeq
with $L_j$ depending on $d$, $\tau$, $L$ and $\alpha$ only. 
\end{lem}
See \lemref{tech} in the appendix for further details on the proof. In the light of this result, the model of Definition \ref{def modele} is thus quite close to the one proposed by \cite{AL17}.} 
\marc{We are ready to define the class of density functions that we study, built upon submanifolds in the class $\cdt$.}

\clement{
\begin{dfn} \label{def: def model stat}
 \marc{Let $1 \leq d \leq D-1$, $\alpha \geq 0$, $\beta \geq 0$,  $\tau > 0$, $L > 0$, $R > 0$ and $0 \leq f_{\min} < f_{\max}$. We define $\Sigma^d_{\alpha,\beta}(\tau,L,f_{\min}, f_{\max} ,R)$, or $\sdab$ for short, as the set of probability measures $P$ on $\bbR^D$} (endowed with its Borel $\sigma$-field) such that
\bitem
\item[(i)] There exists $M_P\in \cdatl$ such that $\supp P = M_P$;
\item[(ii)] There exists a version of the Radon-Nikodym derivative $\frac{dP}{d\mu_{M_P}}$, denoted by $f_P$, that belongs to $\cH^{\beta}(M_P,\bbR)$;
\item[(iii)] This version satisfies $f_{\min} \leq f_P \leq f_{\max}$ and $\|f_P\|_{\mathcal H^{\beta}} \leq R$. 
\eitem
\end{dfn}
Some remarks: {\bf 1)} \marc{By construction, the support of $P \in \Sigma^d_{\alpha,\beta}(\tau,L,f_{\min}, f_{\max} ,R)$ all contains the candidate point $x$, see Definition \ref{def modele}}. {\bf 2)} Condition (i) discards the possibility that $f_P$ is zero on non-null subset of $M$; in particular $f_P$ is non zero around $x$ (but can be zero at $x$ nonetheless). This ensures that $x$ does not lie too far from the data. \marc{An alternate definition is to impose a condition like $P \ll \mu_M$. This leads to the same results in the next sections, but with a slight ambiguity in the choice of $M$}. {\bf 3)} The parameters in subscript or superscript $(d,\alpha,\beta)$ control the rate of convergence of the estimation, while the parameters $(\tau,L,f_{\min}, f_{\max} ,R)$ control the pre-factor in the rates of convergence. For notational simplicity, we sometimes omit them when no confusion can be made.}
\clement{
\subsection{Choice of a loss function and the reach assumption}\label{sec:loss}
For $P \in \sdab$ and a $n$-sample $(X_1,\dots,X_n)$ drawn from $P$, our goal is to recover the value of $f_P(x)$ thanks to an estimator $\wh f(x)$ built on top of the data $(X_1,\dots,X_n)$. We measure the accuracy of estimation by the maximal expected risk or order $p$, for $p \geq 1$, defined by
$$
 \sup_{P \in \sdab} \bbE_{P^{\otimes n}}[|\wh f(x) - f_P(x)|^p]^{1/p}
$$
\marc{We look for an estimator with the smallest possible maximal risk as the number of observations $n$ goes to $\infty$.
We first show that if we let $\tau = 0$, {\it i.e.} if we do not impose a reach condition, then it is impossible to estimate $f(x)$ consistently as $n \to \infty$ for any estimator, thus establishing that the reach assumption $\tau > 0$ unavoidable.}
\begin{thm}\label{thm:reachrisk} \marc{In the setting of Definition \ref{def: def model stat}, if we let $\tau = 0$, the following lower bound holds}
$$
\inf_{\wh f(x)} \sup_{P \in \sdab}\bbE_{P^{\otimes n}} [|\wh f(x)- f_P(x)|^p]^{1/p} \geq \frac12(f_{\max} - f_{\min}) > 0,
$$
where the infimum is taken over all estimators $\wh f(x)$ of $f(x)$.
\end{thm}
The proof is given in  \appref{loss}. This result is in line with a reach condition $\tau > 0$, \marc{a customary necessary condition in a minimax reconstruction in geometric inference, when the manifold is unknown, see \citep{NSW08, Gal12, Bal12, KRW16, AL17} and the references therein}.}

\section{Density estimation at a fixed point} \label{sec:main}

\clement{Recall that we fix the point $x \in \bbR^D$ where we wish to estimate $f$. 
Throughout the section, the symbols $\lesssim$ and $\gtrsim$ denote inequalities up to a constant that, unless specified otherwise, depend on the parameters $d,\alpha,\beta, \tau,L,f_{\min}, f_{\max} ,R$ and $p$. The expression \emph{for $n$ large enough} means for $n$ bigger than a constant that depends on the same parameters.}

\subsection{Main results} \label{sec: dens est}
\clement{
Let $D \geq 2$, $\tau > 0$, $L > 0$, $R > 0$, $0 \leq f_{\min} < f_{\max}$ and $p \geq 1$. Recall that we write $\sdab$ for short for $\sdab(\tau,L,f_{\min},f_{\max},R)$ as defined in Definition \ref{def: def model stat}. The main results of this section are the following
\begin{thm}[Upper bound] \label{thm:upper} \marc{For any $1 \leq d \leq D-1$, $\alpha \geq 0$ and $\beta \geq 0$}, there exists an estimator $\wh f(x)$ \marc{-- explicitly constructed in Section \ref{sec:ker} below --} depending on $\alpha,\beta$ and $d$, such that,  for $n$ large enough,
\marc{
$$
\sup_{P \in  \sdab} \bbE_{P^{\otimes n}}[|\wh f(x) - f_P(x)|^p]^{1/p} \lesssim n^{-\alpha \wedge \beta/(2\alpha \wedge \beta + d)}.
$$
} 
\end{thm}
}
\clement{The estimator of \thmref{upper} is a kernel density estimator that depends on $\alpha,\beta$ and $d$ through the choice of the kernel and its order (in a certain sense specified below), together with its bandwidth. Its analysis si given in \secref{ker}. The estimator is indeed optimal in a minimax sense, as soon as $\alpha \geq \beta$  
}
\clement{
\begin{thm}[Lower bound] \label{thm:lower} \marc{Let $1 \leq d \leq D-1$, $\alpha \geq 0$ and $\beta \geq 0$}. If $L$ and $f_{\max}$ are large enough and if $f_{\min}$ is small enough (depending on $\tau$), then
$$
\liminf_{n \to \infty} n^{\beta/(2\beta+d)} \inf_{\wh f(x)} \sup_{P \in \sdab} \bbE_{P^{\otimes n}}[|\wh f(x) - f_P(x)|^p]^{1/p} \geq C_* > 0$$
where $C_*$ \marc{only} depends on $\tau$ and $R$.
\end{thm}
}
\marc{See \appref{prooflower} for a proof. The rates from \thmref{upper} and \thmref{lower} agree, provided the underlying manifold $M$ is regular enough, namely that $\alpha \geq \beta$. This probably covers most cases of interest in practice. However, when $\alpha < \beta$ the question of optimality remains. We investigate in \secref{uni} below the simpler case $d=1$ and show that it is then possible to achieve the rate $n^{-\beta/(2\beta+1)}$, at the extra cost of learning the geometry of $M$ in a specific sense. 
}
\clement{
\begin{thm}[One-dimensional case]  \label{thm:onedim} \marc{Let $d=1$ and $\beta \geq 0$}. Assume that $f_{\min} > 0$. Then there exists an estimator $\hat f^{\od}(x)$ \marc{-- explicitly constructed in Section \ref{sec:uni} below --}  depending on $\beta$, such that, for any $\alpha \geq 0$ and for $n$ large enough,
$$
\sup_{P \in  \Sigma_{\alpha,\beta}^1} \bbE_{P^{\otimes n}}[|\hat f^{\od}(x) - f_P(x)|^p]^{1/p} \lesssim n^{-\beta/(2\beta + 1)},
$$
\end{thm}
}
\marc{
The estimator described in \thmref{upper} requires the specification of $\alpha$, $\beta$ and $d$, that are usually unknown in practice. We can circumvent this impediment by building an adaptative procedure with respect to these parameters. In \secref{smooth} we adapt to the smoothness parameters $\alpha$ and $\beta$ by implementing Lepski's method \cite{Lep92}; in \secref{dim}, we adapt to $d$ by plugging-in a dimension estimator. We obtain the following result: 
}
\clement{
\begin{thm}[Adaptation]  \label{thm:adapt} Let $\ell \geq 0$. Assume that $f_{\min} > 0$. Then, there exists an estimator $\hat f^{\ad}(x)$ \marc{-- explicitly constructed in Section \ref{sec:dim} below --} depending on $\ell$ such that, for any $\alpha, \beta$ in $[0,\ell]$ and any $1 \leq d \leq D-1$, we have, for $n$ large enough,
$$
\sup_{P \in \sdab} \bbE_{P^{\otimes n}}[|\hat f^{\ad}(x) - f_P(x)|^p]^{1/p} \lesssim \(\frac{\log n}n\)^{\frac{\alpha \wedge \beta}{2\alpha \wedge \beta + d}}. 
$$
\end{thm}
}
\clement{
We were unable to obtain oracle inequalities in the spirit of the Goldenshluger-Lepski method, see \citep{GL08, GL11, GL14}, due to the non-Euclidean character of the support of $f$: our route goes along the more classical approach of \cite{LMS97}. Obtaining oracle inequalities in this framework remain an open problem. 
}

\subsection{Kernel estimation} \label{sec:ker}

Classical nonparametric density estimation methods are based on kernel smoothing \citep{P62, S86}. In this section, we combine kernel density estimation with the minimal geometric features needed in order to recover efficiently their density.
\clement{Since the intrinsic dimension $d$ is not prone to change in this section, we further drop $d$ in (most of) the notation.} The proofs of this section can be found in \appref{ker}.\\

Let $K : \bbR^D \rightarrow \bbR$ be a smooth function vanishing outside the unit ball $B(0,1)$. 
Given an $n$-sample $(X_1,\dots,X_n)$ drawn from a distribution $P$ on $\bbR^D$, we are interested in the behaviour of the kernel estimator 
\beq \label{deffh}
\hfhx = \frac{1}{nh^d} \sum_{j=1}^n K\(\frac{x - X_i}{h}\),~~~~ h >0.
\eeq
\clement{Note that the normalization here is $h^d$ and not $h^D$ as one would set for a classical kernel estimator in $\bbR^D$.}
Our main result is  that $\hfhx$ behaves well when $P$ is supported on a $d$-dimensional submanifold of $\bbR^D$.\\

We need some notations. For $P \in \sdab$ we define
$$
f_h(P,x) = \bbE_{P^{\otimes n}}[\hfhx],
$$
$$
\mathcal B_h(P,x) = f_h(P,x) - f_P(x),
$$
and
$$
\widehat \xi_h(P,x) = \hfhx - f_h(P,x),
$$
that correspond respectively to the mean, bias and stochastic deviation of the estimator $\hfhx$. We also introduce the quantity \clement{
$$
\Omega (h) = \sqrt{\frac{2\omega}{nh^d}} + \frac{\|K\|_\infty}{nh^d}~~~~\text{with}~~~\omega = 4^d \zeta_d \|K\|_\infty^2 f_{\max},
$$}
where $\zeta_d$ is the volume of the unit ball in $\mathbb R^d$. The quantity $\Omega(h)$ will prove to be a good majorant of the stochastic deviations of $\widehat  f_h(x)$. 
The usual bias-stochastic decomposition of $\widehat  f_h(x)$ leads to 
\begin{align} \label{bsd}
\bbE_{P^{\otimes n}} [|\wh f_h(x) - f_P(x)|^p]^{1/p} \leq  |\mathcal B_h(P,x)| + \big( \bbE_{P^{\otimes n}}\big[|\widehat  \xi_h(P,x)|^p\big]\big)^{1/p}. 
\end{align}
We study each term separately. The stochastic term can readily be bounded.
\begin{prp} \label{prp:sd} Let $p \geq 1$. There exists a constant $c_p > 0$ depending on $p$ only such that or any $P \in \sdab$ and any $h < \tau/2$: 
\beq
\(\bbE_{P^{\otimes n}}\big[ |\widehat  \xi_h(P,x)|^p\big]\)^{1/p} \leq c_p \Omega(h) \nonumber.
\eeq
\end{prp}

Now we turn to the bias term. We need certain properties for the kernel $K$. More precisely, we assume that
\begin{ass} \label{ass:kernel} \benum
\item[(i)] $K$ is smooth and supported on the unit ball $B(0,1)$;
\item[(ii)] For any $d$-dimensional subspace $H$ of $\bbR^D$, we have $\int_H K(v)dv = 1$. 
\eenum
\end{ass}

One way to obtain \assref{kernel} is to set $\Lambda(z) = \exp\big(-1/(1 - \|z\|^2)\big)$ for $z \in B(0,1)$ and $\Lambda(x)=0$ otherwise. Since $\Lambda$ is rotationally invariant, its integral is the same over any $d$-subspace $H$ of $\bbR^D$. Thus, with $\lambda_d = \int_{H_0} \Lambda(v)dv$ where $H_0 = \bbR^d\times\{0_{\bbR^{D-d}}\}$, the function $K (z) = \lambda_d^{-1} \Lambda(z)$ is a smooth kernel, supported on the unit ball of the ambient space $\mathbb R^D$ that satisfies \assref{kernel}. In the following, we pick an arbitrary kernel $K$ such that \assref{kernel} is satisfied.

\begin{lem} \label{lem:expans}
For $P \in \sdab$ and any $h < \tau/2$, setting $k = \ceil{\alpha \wedge \beta - 1}$, we have
\beq
f_h(P,x) = f(x) + \sum_{j =1}^{k}  h^j G_j(P,x) + R_h(P,x), \eeq
with $| G_j(P,x) | \lesssim 1 $ and $| R_h(P,x) | \lesssim h^{\alpha\wedge \beta}$.
\end{lem}
The existence of such an expansion allows, by carefully choosing the kernel, to cancel the intermediate terms. Starting from a kernel $K$ satisfying \assref{kernel}, we recursively define a sequence of smooth kernels $(K^{(d, \ell)})_{\ell \geq 1}$, simply denoted by $K^{(\ell)}$ in this section, with support in $B(0,1)$ as follows (see \figref{kl}). For $z \in \bbR^D$, we put
\beq \label{kell}
\begin{cases} K^{(1)}(z) = K(z) \\
K^{(\ell + 1)}(z) = 2^{1+d/\ell} K^{(\ell)}(2^{1/\ell} z) - K^{(\ell)}(z) ~~~\forall \ell \geq 1.
\end{cases}
\eeq
A few remarks can be made: {\bf 1)} In a classical kernel density estimation framework, the integer $\ell -1$ plays the role of the order of the kernel. {\bf 2)} The assumption that $K$ is compactly supported is seemingly quite strong. This is a way to make sure that the support of $y \mapsto K\((y - x)/h\)$ is within the injectivity ball of the map $\exp_x$ for any $h < \pi\tau$. {\bf 3)} The construction of $K$ is simply an example of a Richardson's extrapolation as coined by \citet{Ric11}. {\bf 4)} This construction somewhat differs from the classical constructions than can be found in textbooks such as \citet{Tsy08}. There is one practical reason: we require that all the kernels satisfy \assref{kernel}; another reason that appears to be more intrinsically related to our model: \marc{because the Euclidean distance is only a second order approximation of the Riemannian distance on $M$, defining a kernel through orthogonality relations with respect to a family of polynomials is not sufficient in our framework}.

\begin{figure}[ht]
\centering
\includegraphics[scale=0.55]{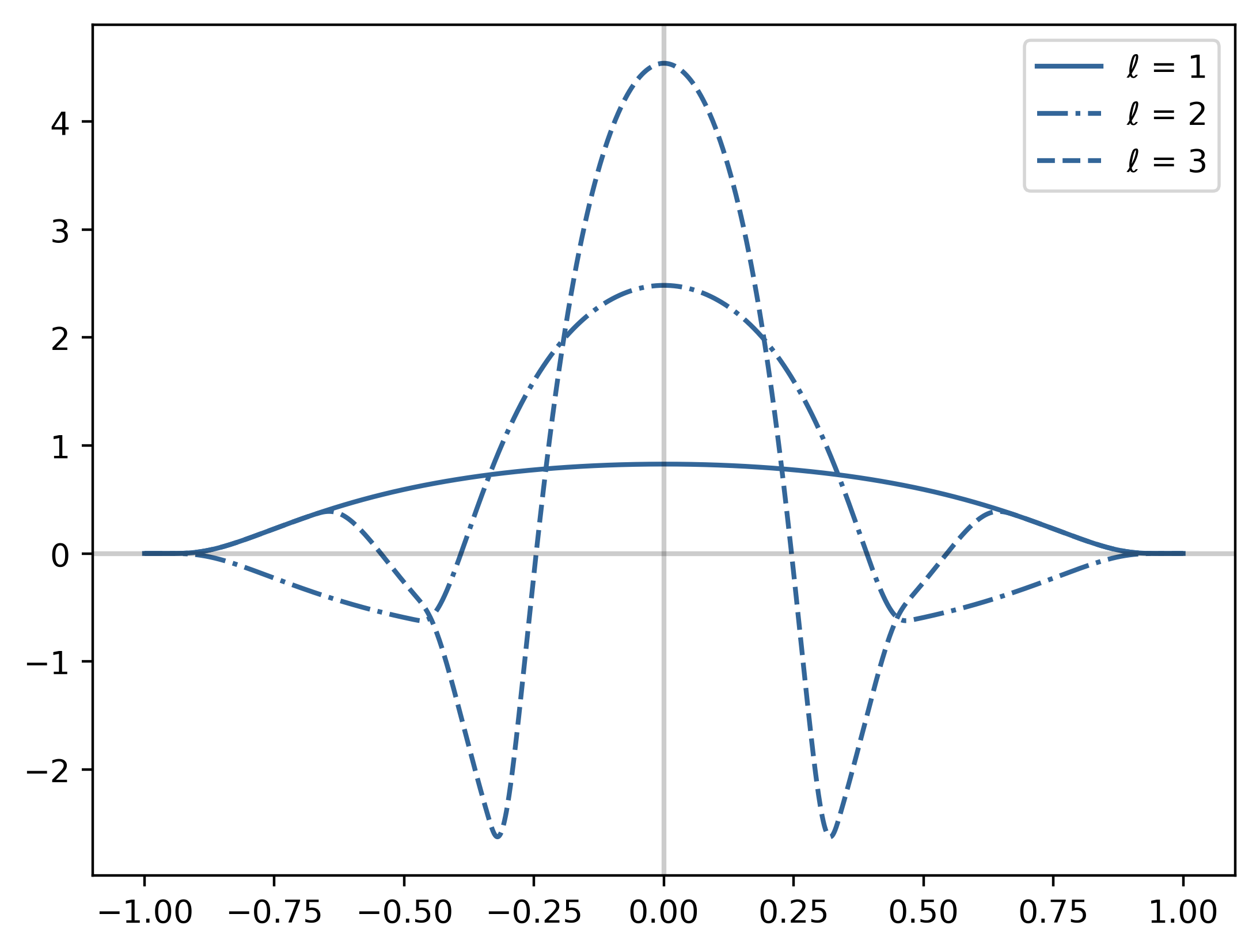}
\caption{\small Plots of the kernel $K^{(d,\ell)}$ for $d = 1$ and $\ell =1,2,3$.}
\label{fig:kl}
\end{figure}
\begin{prp} \label{prp:bias}
Let $\ell$ be an integer greater than $\alpha$ and $\beta$, and let $K^{(\ell)}$ be the kernel defined in \eqref{kell} starting from a kernel $K$ satisfying \assref{kernel}. Then, for any $P \in \sdab$ and  any bandwidth $h < \tau/2$, the estimator $\widehat  f_h(x)$ defined as in \eqref{deffh} using $K^{(\ell)}$ is such that
\beq \label{biasx}
|\mathcal B_h(P,x)| \lesssim h^{\alpha \wedge \beta}.
\eeq 
\end{prp}
\clement{Patching together \prpref{sd} and \prpref{bias} yields 
$$
\sup_{P \in \sdab} \bbE_{P^{\otimes n}} [|\wh f_h(x) - f_P(x)|^p]^{1/p} \lesssim \Omega(h) + h^{\alpha\wedge\beta}, 
$$
for any $h < \tau/2$. Therefore the estimator $\wh f(x) = \wh f_h(x)$ specified with $h = n^{-1/(2\alpha\wedge\beta + d)}$ indeed satisfies the conclusion of \thmref{upper}, for $n$ large enough. 
}

\subsection{\marc{A minimax estimator in the case $d=1$ that covers the case $\alpha < \beta$}} \label{sec:uni}

\marc{
The gap we observe between the two rates in \thmref{upper} and \thmref{lower} leads to the following question: does the regularity $\alpha \geq 0$ of $M$ have a genuine limiting effect in the estimation of $f(x)$, or does it rather reveal a weakness of the estimator described in \secref{ker}. We do not have a definitive answer to this question except for $d=1$ {\it i.e.} when $M$ is a closed curve in an Euclidean space. We can then show that the parameter $\alpha$ does not interfere at all with the density estimation.} The proofs of this section can be found in \appref{uni}. \\

If $d=1$, any submanifold $M$ in $\cC_{1}(\tau)$ is a closed smooth injective curve that can be parametrized by a unit-speed path $\gamma_M : [0, L_M] \rightarrow \bbR^D$ with $\gamma_M(0) = \gamma_M(L_M)$ and with $L_M = \vol M$ being the length of the curve. In that case, the volume density function is trivial.

\begin{lem} \label{lem:dim1}
For $M \in \cC_1(\tau)$, for any $z \in M$ and any $v \in T_z M$, we have $\det g^z(v) = 1$. 
\end{lem}

Thanks to \lemref{dim1}, the estimator proposed by \cite{Pel05} takes a simpler form, which we will try to take advantage of. Indeed, in the representation \eqref{pelest} of \cite{Pel05}, only $d_M$ remains unknown.  We now show how to efficiently estimate $d_M$ thanks to the Isomap method as coined by \cite{TDL00}. The analysis of this algorithm essentially comes from \cite{Bal00} and is pursued in \cite{AG19}, but the bounds obtained there are manifold dependent. We thus propose a slight modification of their proofs in order to obtain uniform controls over $\cC_1(\tau)$, and make use of the simplifications coming from the dimension $1$. Indeed, for $d=1$, we have the following simple and explicit formula for the intrinsic distance on $M$: 
\beq
d_M(\gamma_M(s), \gamma_M(t)) = |t-s| \wedge (L_M - |t-s|)~~~\forall s,t \in [0, L_M]. \nonumber
\eeq

\newcommand{\wge}{\wh G_\ve(x)}

The Isomap method can be described as follows: let $\ve > 0$, and let $\cG_\ve$ be the $\ve$-neighbourhood graph built upon the data $(X_1,\dots,X_n)$ and $x$ --- namely, $\cG_{\ve} = (V,E)$ where $V = \(x,X_1,\dots,X_n\)$, and where $E = \{(y,z) \in V~|~\|y-z\| \leq \ve\}$. For a path in $\cG_\ve$ (meaning: a sequence of adjacent vertices) $s = (p_0, \dots , p_m)$, we define its length as $L_s = \|p_1 - p_0\|+\dots+\|p_m - p_{m-1}\|$. The distance between $x$ and a vertice $y$ in the graph $\cG_\ve$ is then defined as
\beq
\wh d_\ve(x,y) = \min\big\{L_s~|~s~\text{path in}~\cG_\ve~\text{connecting $x$ to $y$}\big\},
\eeq
and we set this distance to $\infty$ if $x$ and $y$ are not connected.  We are now ready to describe our estimators $\hat f^{\od}(x)$. For any $h, \ve > 0$, we set

\beq \label{kdeuni}
\hat  f_h^{\od}(x) =\frac{1}{nh} \sum_{i=1}^n K^{\od}\big( \wh d_\ve(x,X_i)/h\big),
\eeq 
for some kernel $K^{\od} : \bbR \rightarrow \bbR$. Notice that the kernel $K^{(1,\ell)}(\cdot) : \bbR^D \rightarrow \bbR$ defined in \secref{ker} starting from kernel $K = \lambda_1^{-1} \Lambda$ can be put in the form $K^{(1,\ell)}(y) = K^{(1,\ell)}(\|y\|)$ with $K^{(1,\ell)}(\cdot)$ denoting thus (with a slight abuse of notation) both functions starting from either $\bbR$ or $\bbR^D$. We choose this kernel in the next statement.
\clement{
\begin{prp} \label{prp:kerd1} Assume that $f_{\min} > 0$. The estimator defined in \eqref{kdeuni} above and specified with $K^{\od} = K^{(1,\ell)}(\cdot)$ satisfies the following property:  for any $\beta \in [0,\ell]$ and any $\alpha \geq 0$, we have
\beq
\sup_{P \in \Sigma_{\alpha,\beta}^1} \bbE_{P^{\otimes n}}[|\hat f^{\od}_h(x) - f_P(x)|^p]^{1/p} \lesssim \frac{\ve^2}{h^2} + \Omega(h) + h^\beta + \frac{1}{nh}~~~~\text{with}~~~\ve = \frac{32(p+1)}{f_{\min} } \times \frac{\log n}n, \nonumber
\eeq
for $h < \tau/4$ and $n$ large enough.
\end{prp}
}
The proof of \thmref{onedim} readily follows from \prpref{kerd1} using the estimator $\hat f^{\od} = \hat  f^{\od}_h$ with $h = n^{-1/(2\beta+1)}$.

\subsection{Smoothness adaptation} \label{sec:smooth}

We implement Lepski's algorithm, following closely \cite{LMS97} in order to automatically select the bandwidth from the data $(X_1,\ldots, X_n)$.  We know from Section \ref{sec:ker} that the optimal bandwidth on $\sdab$ is of the form $n^{-1/(2\alpha \wedge \beta)+d}$. Hence, without prior knowledge of the value of $\alpha$ and $\beta$, we can restrict our search for a bandwidth in a bounded interval of the form $[h^-, 1]$ discretized as follow
\beq
\bbH = \big\{2^{-j},~~\text{for}~~0 \leq j \leq  \log_2 (1/h^-)\big\} \nonumber
\eeq
We choose to pick
$$
h^- = \(\frac{\|K\|_\infty}{2\omega}\)^{1/d} \frac1{n^{1/d}};
$$
this bandwidth is always smaller than the optimal bandwidth $n^{-1/(2\alpha \wedge \beta + d)}$ on $\sdab$ for $n$ large, and is such that $\Omega(h) \leq 2\sqrt{2\omega/(nh^d)}$ for all $h \geq h^-$. For $h,\eta \in \bbH$, we introduce the following quantities:
\begin{align}
\lambda(h) &= 1 \vee \sqrt{ \Theta d \log(1/h)}, \nonumber \\
\psi(h,\eta) &= \Omega(h) \lambda(h) +  \Omega(\eta) \lambda(\eta) \label{defpsi}
\end{align}
where $\Theta$ is a positive constant (to be specified). For $h \in \bbH$ we define the subset of bandwidths $\bbH(h) = \{\eta \in \bbH,~\eta \leq h\}$ 
The selection rule for $h$ is the following:
\beq  \label{defhh}
\widehat  h(x) = \max \Big\{ h \in \bbH~|~\forall \eta \in \bbH(h), ~|\widehat  f_h(x) - \widehat  f_\eta(x)| \leq \psi(h,\eta) \Big\}, \nonumber
\eeq
and we finally consider the estimator
\begin{equation} \label{eq:defadapt}
\widehat  f(x) = \widehat  f_{\widehat  h(x)}(x).
\end{equation}
where we recall that $f_h$ is defined at \eqref{deffh}. 
\clement{
\begin{prp} \label{prp:adapt} Assume $\Theta > p$. Let $\ell \in \bbN$, and let $\wh f(x)$ be the estimator defined in \eqref{eq:defadapt} using $K^{(\ell)}$ originated from a kernel $K$ satisfying \assref{kernel}. Then, for any $\alpha,\beta \in [0,\ell]$, we have, for $n$ large enough
$$
\sup_{P \in \sdab}\bbE_{P^{\otimes n}}[|\widehat  f(x) - f_P(x)|^p]^{1/p} \lesssim \(\frac{\log n}{n}\)^{\alpha \wedge \beta/(2\alpha \wedge \beta+d)}.
$$
\end{prp}
}
The proof of \prpref{adapt} can be found in \appref{smooth}.
Some remarks: {\bf 1)} \prpref{adapt} provides us with a classical smoothness adaptation result in the spirit of \citep{LMS97}: the estimator $\widehat  f$ has the same performance as the estimator $\widehat f_h$ selected with the optimal bandwidth $n^{-1/(2\alpha \wedge \beta+d)}$, up to a logarithmic factor
on each model $\sdab$ without the prior knowledge of $\alpha \wedge \beta$ over the range $[0, \ell]$. {\bf 2)} The extra logarithmic term is the unavoidable payment for the Lepski-Low phenomenon \cite{L90, Lo92} when recovering a function in pointwise or in a uniform loss.

\subsection{Simultaneous adaptation to smoothness and dimension} \label{sec:dim}
\clement{
The estimators considered in \thmref{upper} or \prpref{adapt} heavily rely on the intrinsic dimension $d$ through the choice a of kernel satisfying \assref{kernel}, through the normalization $h^d$ and either through the choice of an optimal bandwidth $h$, or the selection procedure \eqref{defpsi}-\eqref{eq:defadapt}.} \marc{We now show how to adapt to $d$ considered as an unknown and nuisance parameter.}The proofs of this section can be found in \appref{dim}.

We redefine all the quantities introduced before as now depending on $d$. Namely, for $h,\eta > 0$, and a given family of kernel $K(d;\cdot)$, we set
\begin{align*}
\wh f_h(d;x) &= \frac1{nh^d} \sum_{i=1}^n K(d; (X_i - x)/h). \label{deffhd} \\
\Omega(d;h) &= \sqrt{\frac{2\omega_d}{nh^d}} + \frac{\|K(d;\cdot)\|_\infty}{nh^d}~~~\text{with}~~~ \omega_d = 4^d \zeta_d \|K(d;\cdot)\|_\infty^2 f_{\max},  \\
\lambda(d;h) &= 1\vee\sqrt{\Theta d \log(1/h)}, \\
\psi(d;h,\eta) &= v(d;h) \lambda(d;h) + v(d;\eta) \lambda(d;\eta), \\
h^-_d &= (\|K(d;\cdot)\|_\infty / 2\Omega_d)^{1/d}n^{1/d}, \\
\bbH_d &= \{2^{-j}, ~\text{for}~~ 0 \leq j \leq \log_2(1/h_d^-)\},
\end{align*}
where $\Theta$ is a constant. We also define
\beq \label{whh}
\wh h(d; x) = \max\{h \in \bbH_d ~|~ \forall \eta \in\bbH_d(h),~~|\wh f_h(d;x) - \wh f_\eta(d;x)| \leq \psi(d;h,\eta)\}
\eeq
with $\bbH_d(h) = \{\eta \in \bbH_d, ~\eta \leq h\}$. We are now left with the choice of kernel family $K(d;\cdot)$. For any $1 \leq d \leq D-1$ and $h > 0$, we define
\beq
K^{(1)}(d; x) = \lambda_d^{-1} \Lambda(x) \nonumber
\eeq 
where $\Lambda$ and $\lambda_d$ have been introduced in \secref{ker}. We then pick an integer $\ell \in \bbN$ and choose 
\beq
K(d;\cdot) = K^{(d,\ell)}(d;\cdot) \label{kdl}
\eeq
where $K^{(d,\ell)}(d;\cdot)$ is defined by recursion in \eqref{kell} starting from the kernel $K^{(1)}(d;\cdot)$. 

We assume that we have an estimator $\wh d$ of the dimension $d$ of $M$ with values in $\{1,\dots, D\}$. More precisely, we require the following property
\begin{ass} \label{ass:dim}
For any $1 \leq d < D$ and all real numbers $\alpha,\beta \in [0,\ell]$, we have 
$$\sup_{P \in \sdab} P^{\otimes n}\(\wh d \neq d\) \lesssim n^{-3p/2}.$$
\end{ass}

If we are given such a estimator of the dimension $d$, then we can built a estimator that adapts to this parameter.
\clement{
\begin{prp}\label{prp:dimadapt} Let $\wh f(x) = \wh f_{\wh h}(\wh d, x)$ built with the kernel family \eqref{kdl}, where $\wh d$ is a estimator satisfying \assref{dim} and where $\wh h = \wh h(\wh d,x)$ is defined at \eqref{whh}. Then, for any $1 \leq d \leq D -1$, and any $\alpha,\beta \in [0,\ell]$, we have, for $n$ large enough
$$
\sup_{P \in \sdab}\bbE_{P^{\otimes n}}[|\widehat  f(x) - f_P(x)|^p]^{1/p} \lesssim \(\frac{\log n}{n}\)^{\alpha \wedge \beta/(2 \alpha \wedge \beta + d)} \nonumber
$$
\end{prp}
It only remains to show that there exists an estimator $\wh d$ satisfying \assref{dim} to obtain \thmref{adapt}. 
}
There are various way to define such an estimator, see \cite{FSA07} or even \cite{KRW16} where an estimator with super-exponential minimax rate on a wide class of probability measures is constructed. For sake of completeness and simplicity, we will mildly adapt the work of \cite{FSA07} to our setting. The resulting estimator will behave well as soon as we add the assumption that $f_{\min} > 0$. 

\begin{dfn} For a probability measure $P$, we write $P_\eta = P(B(x,\eta))$ for any $\eta > 0$, and $\wh P_\eta = \wh P_n(B(x,\eta))$ where $\wh P_n = n^{-1}\sum_{i = 1}^n \delta_{X_i}$ denotes the empirical measure of the sample $(X_1,\ldots, X_n)$. Define
\beq
\wh \delta_\eta = \log_2 \wh P_{2\eta} - \log_2 \wh P_\eta, \nonumber
\eeq
and set $\wh \delta_\eta = D$ when $\wh P_\eta = 0$. We define $\wh d_\eta$ to be the closest integer of $\{1,\dots, D\}$ to $\wh \delta_\eta$, namely $\wh d_\eta = \floor{\wh \delta_\eta + 1/2}$. 
\end{dfn}


\begin{prp} \label{prp:dimest} Assume that $f_{\min} > 0$. Then, for any $1 \leq d \leq D-1$, and any $\alpha,\beta \geq 0$, the estimator $\wh d = \wh d_\eta$ for $\eta = n^{-1/(2D+2)}$ verifies for $n$ large enough 
$$\sup_{P \in \sdab} P^{\otimes n} \( \wh d \neq d \) \leq 4 \exp\(-2n^{1 - (d+1)/(D+1)}\).$$
\end{prp}

\section{Numerical illustration} \label{sec:simu}

In this section we propose a few simulations to illustrate the results presented above. The goal is two-fold
\bitem
\item To highlight the rate obtained in \thmref{upper} using estimator $\wh f_h$, in the case where $\beta \leq \alpha$, on arbitrary submanifold and for a carefully chosen bandwidth $h$;
\item To show the computational feasability and performance of estimator $\hat f^{\ad}$ described in \secref{smooth}.
\eitem 
For the sake of visualisation and simplicity, we focus on two typical examples of submanifold of $\bbR^D$, namely non-isometric embeddings of the flat circle $\bbT^1 = \bbR / \bbZ$ and of the flat torus $\bbT^2 = \bbT^1 \times \bbT^1$.  In particular, these embeddings will be chosen in such way that their images, as submanifolds of $\bbR^D$, are not homogeneous compact Riemannian manifolds, so that the work of \cite{KNP12} for instance cannot be of use here.

For a given embedding $\Phi : N \rightarrow M \subset \bbR^D$ where $N$ is either $\bbT^1$ of $\bbT^2$, we construct absolutely continuous probabilities on $M$ by pushing forward probability densities of $N$ w.r.t. their volume measure. Indeed, if $Q = g \cdot \mu_N$, the push-forward measure $P = \Phi_* Q$ has density $f$ with respect to $\mu_M$ given by
\beq \label{formulaf}
\forall z \in M,~~~ f(z) = \frac{ g(\Phi^{-1}(z)) }{\left|\det d\Phi(\Phi^{-1}(z))\right|} 
\eeq
where the determinant is taken in an orthonormal basis of $T_{\Phi^{-1}(x)} N$ and $T_x M$, so that, if $\Phi$ is chosen smooth enough, $f$ has the same regularity as $g$. If $\Phi$ is an embedding of $\bbT^1$, we simply have $\left|\det d\Phi(y)\right| = \|\Phi'(y)\|$ for all $y \in \bbT^1$. If now $\Phi$ maps $\bbT^2$ to $M$, we have
\begin{align} \label{detform}
\left|\det d\Phi(y)\right| &=  \sqrt{\det \inner{d\Phi(y)[e_i]}{d\Phi(y)[e_j]}_{1\leq i,j \leq 2}} \nonumber \\
&=  \sqrt{ \left\|d\Phi(y)[e_1]\right\|^2 \left\| d\Phi(y)[e_2] \right\|^2 - \inner{d\Phi(y)[e_1]}{d\Phi(y)[e_2]}^2}
\end{align}
where $(e_1, e_2)$ is an orthonormal  basis of $\bbR^2 \simeq T_y \bbT^2$. 

Strictly speaking, the probability measures $P$ exhibited below are not elements of the models $\sdab$, but we know that they locally coincide with some $\widetilde P \in \sdab$ around our candidate point $x$, meaning that
$$
P_{| B(x,r)} = \widetilde P_{|B(x,r)}~~~\text{for some}~~r>0.
$$
This ensures that all the results displayed in \secref{main} hold for $P$ --- see \remref{local} for a discussion on the local character of our setting.

\subsection{An example of a density supported by a one-dimensional submanifold}

Let $\beta \in \bbN^*$ and define the following function for $v \in [-1/2,1/2]$
\beq \label{gs}
g_\beta(v) = C_\beta \times \( 1 - (-2v)^{\beta}\) \ind_{[-1/2,0)}(v)+ C_\beta \(1 - (2v)^{\beta+1}\)\ind_{[0,1/2]}.
\eeq
where $C_\beta$ is an explicit normalisation constant. The function $g_\beta$ is positive and $\int_0^1g_\beta(v)dt=1$; it defines a probability density over [-1/2,1/2]. Also, because the $(\beta-1)$-th derivative of $g_\beta$ is $1$-Lipschitz, but its $\beta$-derivative is discontinuous at $v = 0$, the function $g_\beta$ is $\beta$-H\"older but {\it not} $(\beta+\ve)$-H\"older for any $\ve > 0$. See \figref{gbeta} for a few plots of the functions $g_\beta$. 
\begin{figure}[ht!]
\centering
\includegraphics[scale=0.55]{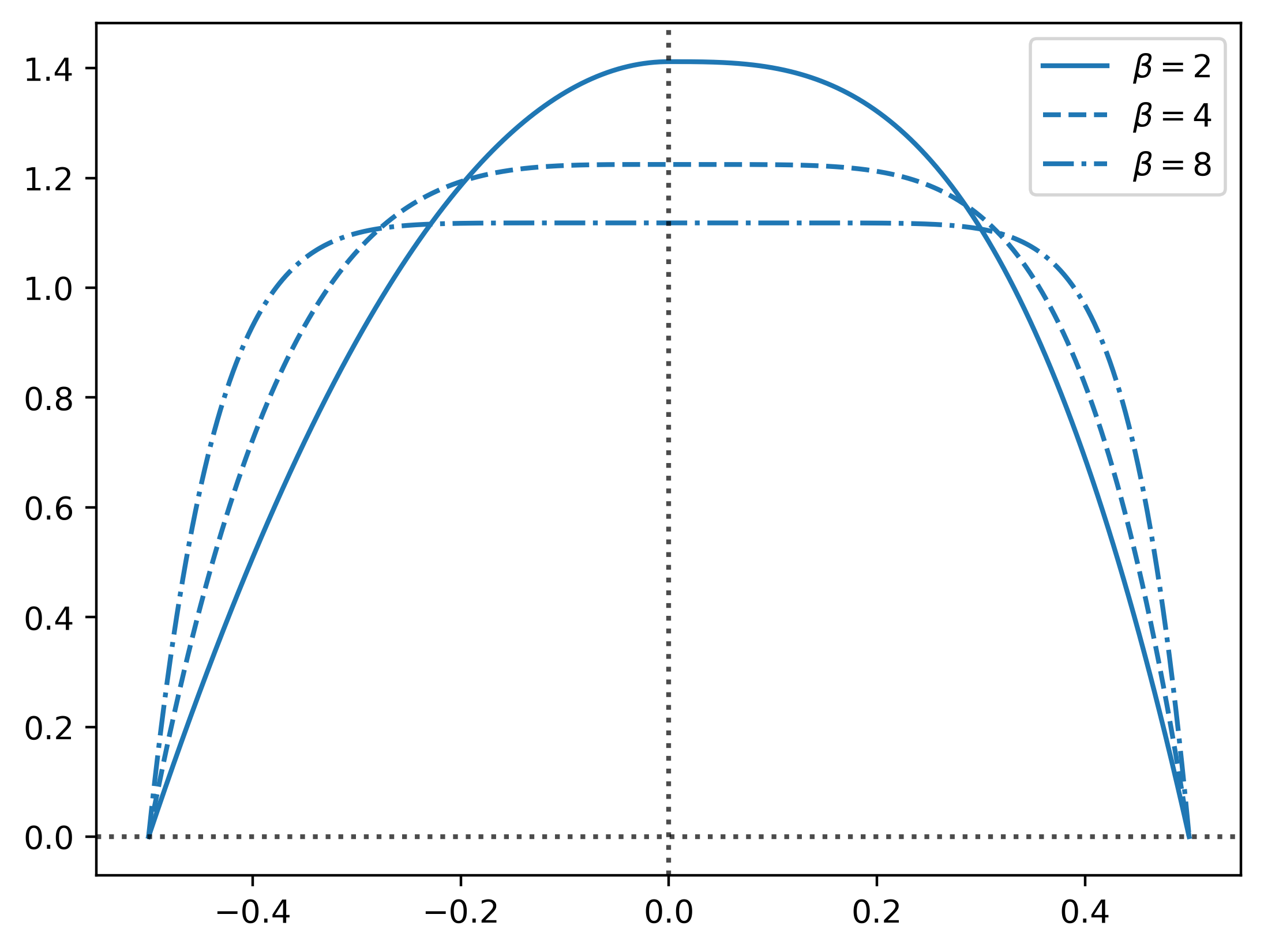}
\caption{\small Plots of the densities $g_\beta$ for $\beta = 2,4,8$.}
\label{fig:gbeta}
\end{figure}
We next consider the parametric curve 
\beq
\Phi : \begin{cases} \bbT^1 \rightarrow \bbR^2 \\
t \mapsto \(\cos(2 \pi v) + a \cos( 2 \pi \omega v), \sin(2 \pi v) + a \sin(2 \pi \omega v) \). \nonumber
\end{cases}
\eeq
Short computations show that $\Phi$ is indeed an embedding as soon as $a \omega < 1$, in which case $M = \Phi(\bbT^1)$ is indeed a smooth compact submanifold of $\bbR^2$. For the rest of this section, we set $a = 1/8$ and $\omega = 6$. See \figref{vard1} for a plot of $M$ with these parameters. 
\begin{figure}[ht!]
\centering
\includegraphics[scale=0.5]{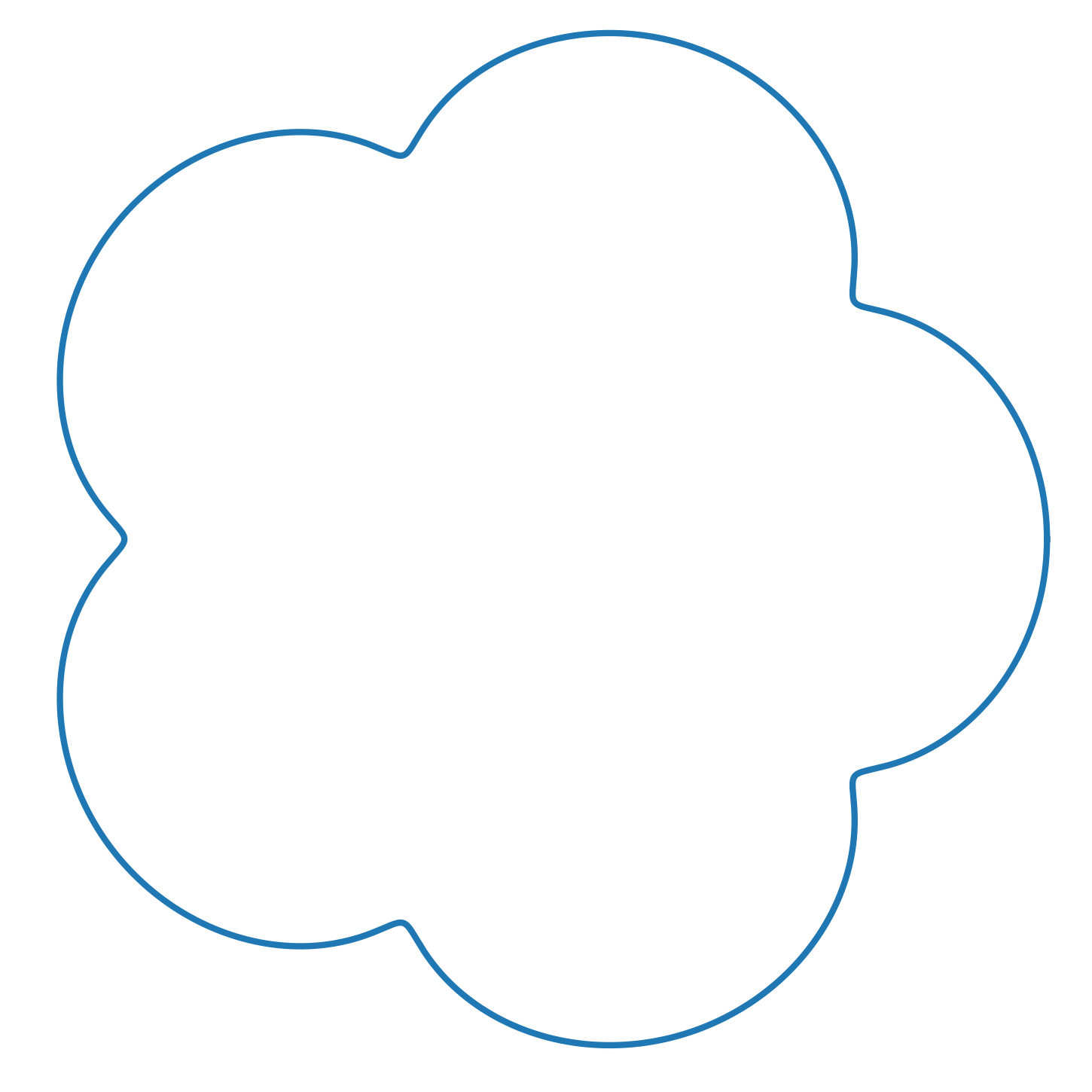} \hspace{1cm}
\includegraphics[scale=0.5]{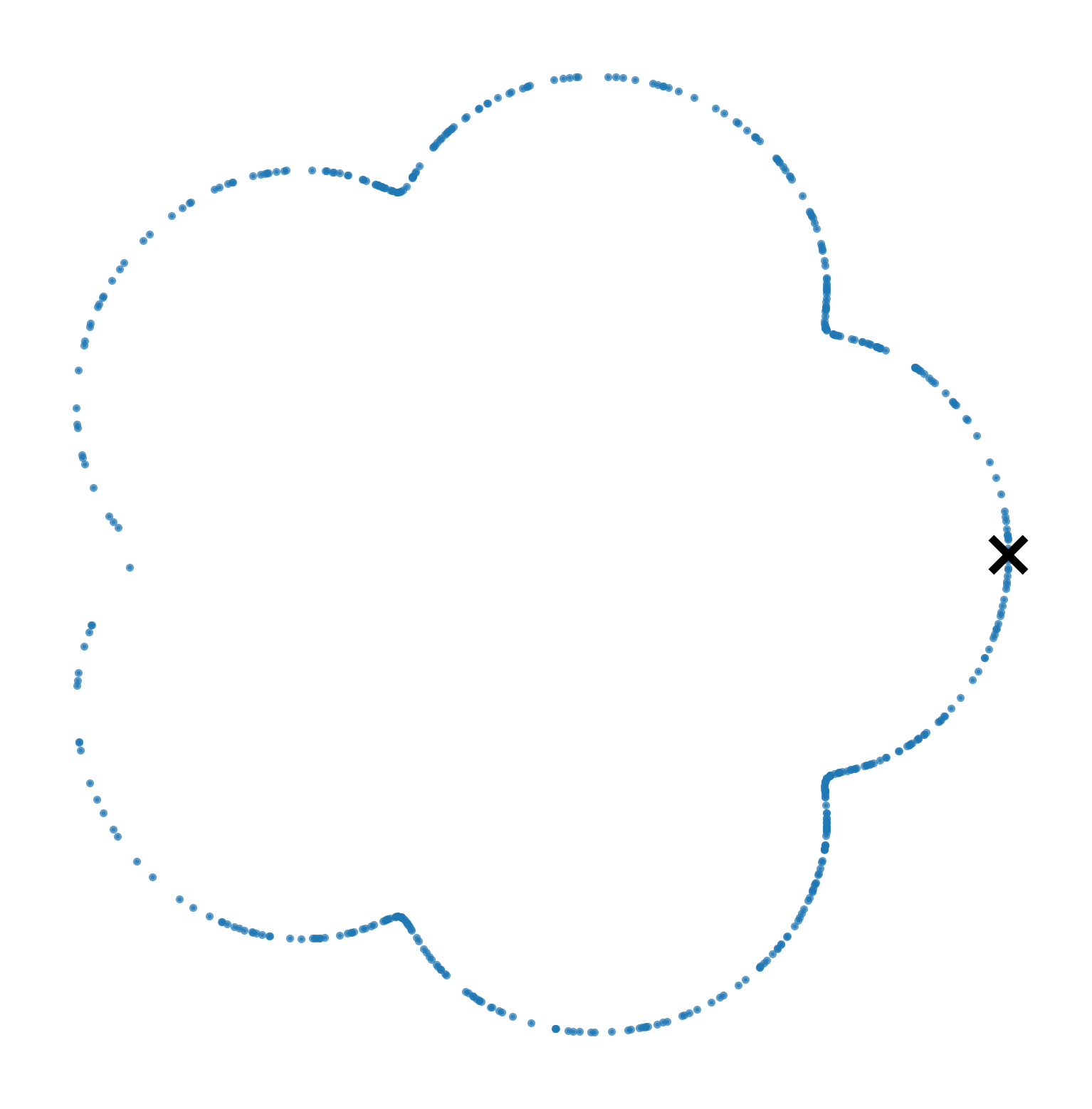}
\captionsetup{justification = centering}
\caption{
\small Plot of the submanifold M (Left) for parameters $a = 1/8$ and $\omega = 6$, and 500 points sampled independently from $\Phi_* g_\beta \cdot \mu_{\bbT^1}$ for $\beta = 3$ (Right). The black cross denotes the point $x = \Phi(0)$.
}
\label{fig:vard1}
\end{figure}
We are interested in estimating the density $f_\beta$ with respect to $d\mu_M$ of the push-forward measure $P_\beta = \Phi_* g_\beta \cdot \mu_{\bbT^1}$, at point $x = (1+a,0) \in M$. We use formula \eqref{formulaf} to compute $f_\beta(x)$: We have $\Phi^{-1}(x) = 0$ and $\|\Phi'(0)\| = 2\pi(1 + a\omega)$ hence
\beq
f_\beta(x) = \frac{C_\beta}{2\pi(1+a\omega)} \;\;\text{at}\;\;x = (1+a,0). \nonumber
\eeq

Our aim here is to provide an empirical measure for the convergence of the risk $n \mapsto \bbE_{{P_\beta}^{\otimes n}}[|\wh f_h(x) - f_\beta(x)|^p]^{1/p}$ when $h$ is tuned optimally (in an oracle way). We pick $p=2$. Our numerical procedure is detailed in Algorithm 1 below, and the numerical results are presented in \figref{msed1}.

\begin{algorithm}[ht!]  \label{algo: algo 1}
	\caption{MSE rate of convergence estimation}
	\begin{algorithmic}[1] 
		\State Provide integers $\beta \geq 1$ and $\ell \geq \beta$. 
		\State Set a grid of increasing number of points $\bn = (n_1,\dots, n_k) \in \bbN^k$ and a number of repetition $N$. 
		\For {$n_i \in \bn$} 
			\State Sample $n_i$ points independently from $P_\beta$,
			\State Compute $\widehat  f_{h_i}(x)$ with kernel $K^{(\ell)}$ and bandwidth $h_i = n_i^{-1/(2\beta+1)}$,
			\State Compute the square error $( \widehat  f_{h_i}(x) - f_\beta(x))^2$,
			\State Repeat the three previous steps $N$ times,
			\State Average the errors to get a Monte-Carlo approximation $\wh R(n_i)$ of $\bbE_{{P_\beta}^{\otimes n_i}}[|\wh f_h(x) - f_\beta(x)|^2]$. 
		\EndFor
		\State Perform an Ordinary Least Square Linear Regression on the curve $\log n_i \mapsto \log \wh R(n_i)$. 
		\State \Return The coefficient of the linear regression. 
	\end{algorithmic} \label{algo:main}
\end{algorithm}

\subsection{An example of a density supported by a two-dimensional submanifold}

We consider a non-isometric embedding of the flat torus $\bbT^2$. We first construct a density function. For and integer $\beta \geq 1$, define
\beq
G_\beta : (v,u) \in [-1/2,1/2]^2 \mapsto g_\beta(v) g_\beta(u) \noindent 
\eeq
where $g_\beta$ is defined as in \eqref{gs}. Obviously, $G_\beta$ defines a density function on $\bbT^2$ that is $\beta$-H\"older (but not $(\beta+\ve)$-H\"older for any $\ve > 0$).

\begin{figure}[ht]
\centering
\includegraphics[scale=0.7]{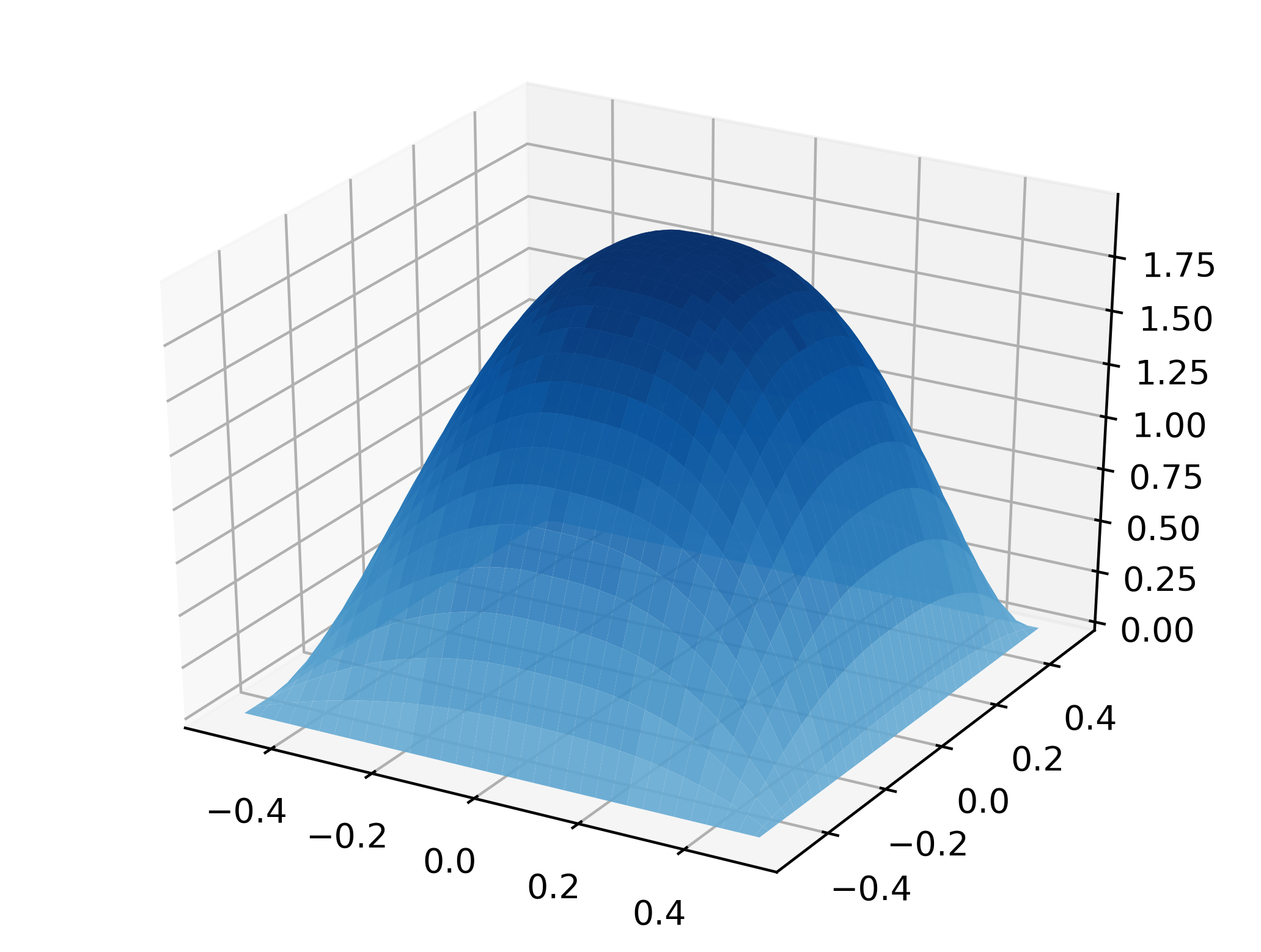}
\caption{\small Plot of the probability density function $G_\beta$ for $\beta = 3$.}
\label{fig:ggbeta}
\end{figure}

We next consider the parametric surface 
\beq \nonumber
\Psi : \begin{cases}
\bbT^2 \rightarrow \bbR^3 \\
(v,u) \mapsto \begin{pmatrix}
(b + \cos(2\pi v))\cos(2\pi u) + a \sin(2\pi\omega v) \\
(b + \cos(2\pi v))\sin(2\pi u) + a \cos(2\pi\omega v) \\
\sin(2\pi v) + a \sin(2\pi\omega u)
\end{pmatrix}
\end{cases}
\eeq
for some $a,b, \omega \in \bbR$. In the remaining of the section, we set $a = 1/8$, $b = 3$ and $\omega = 5$. We show that $\Phi$ indeed defines an embedding. See \figref{vard2} for a plot of the submanifold $M = \Psi(\bbT^2)$. For an integer $\beta \geq 1$, we denote by $F_\beta$ the density of the push forward measure $Q_\beta = \Psi_* G_\beta \cdot \mu_{\bbT^2}$ with respect to the volume measure $\mu_M$. Let $x = (b+1,a,0)$ be the image of $0$ by $\Psi$. 
\begin{figure}[ht!]
\centering
\includegraphics[scale=1]{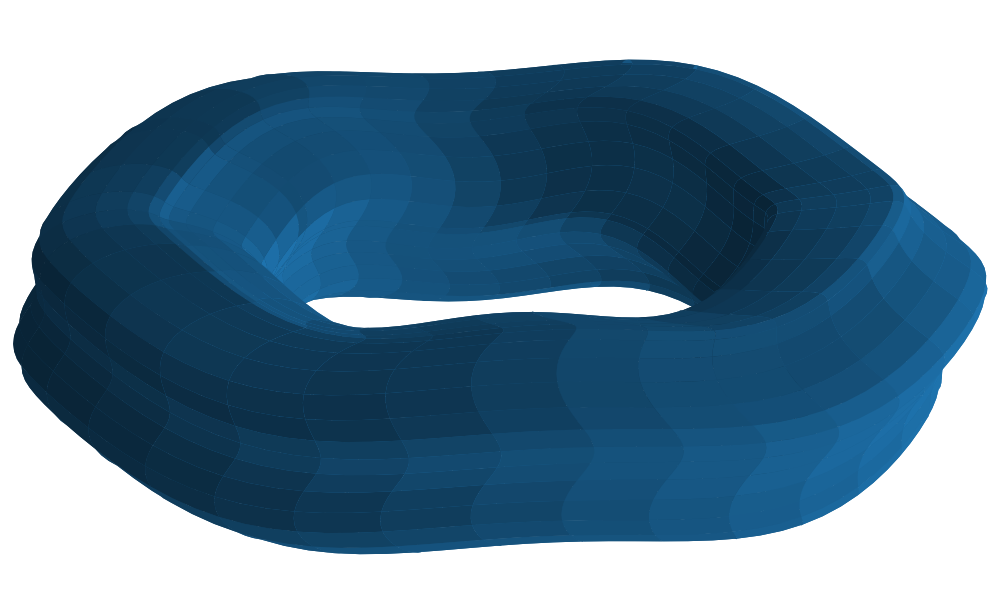} \hspace{1cm}
\includegraphics[scale=1]{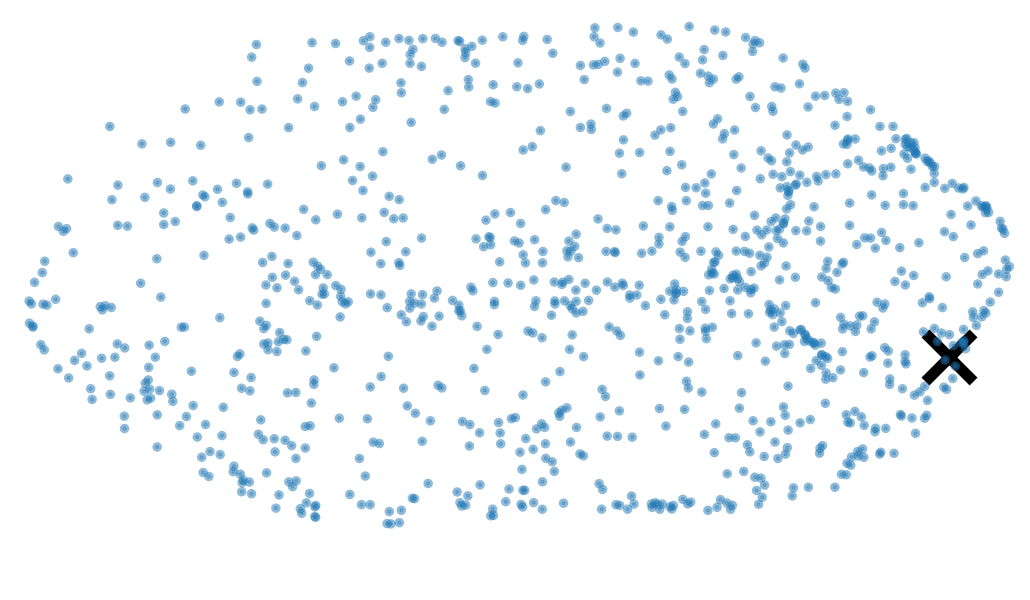}
\captionsetup{margin = 1cm}
\caption{\small Plot of the submanifold $M$ (Left) for parameters $a = 1/8$, $b = 3$ and $\omega = 6$, and 500 points sampled independently from $\Psi_* G_\beta \cdot\mu_{\bbT^2}$ with $\beta = 3$ (Right). The black cross marks the point $x$.}
\label{fig:vard2}
\end{figure}
Simple calculations show that the differential of $\Psi$ at $0$ evaluated at $e_1 = (1,0) \in T_0 \bbT^2$ and $e_2 = (0,1) \in T_0 \bbT^2$ is equal  to respectively $d\Psi(0)[e_1] = 2\pi(a\omega,0,1)$ and $d\Psi(0)[e_2] = 2\pi (0, b+1,a\omega)$. Hence formula \eqref{detform} yields 
\beq \label{detphi}\det d\Psi(0) = (2\pi)^2 \(\(1 + a^2 \omega^2\)\((b+1)^2 + a^2 \omega^2\) - a^2 \omega^2\)^{1/2} \eeq
and we obtain
\beq
F_\beta(x) = \frac{C_\beta^2}{4\pi^2} \(\(1 + a^2 \omega^2\)\((b+1)^2 + a^2 \omega^2\) - a^2 \omega^2\)^{-\frac12}. \nonumber
\eeq

In the same way as in the previous section, we aim at providing an empirical measure for the rate of convergence of the risk $\bbE_{{P_\beta}^{\otimes n}}[|\wh f_h(x) - f_\beta(x)|^2]$ when $h$ is suitably tuned with respect to $n$ and $\beta$. This is done using again \algoref{main}. The results are presented in \figref{msed1}. 

\begin{figure}[ht!]
\centering
\includegraphics[scale=0.5]{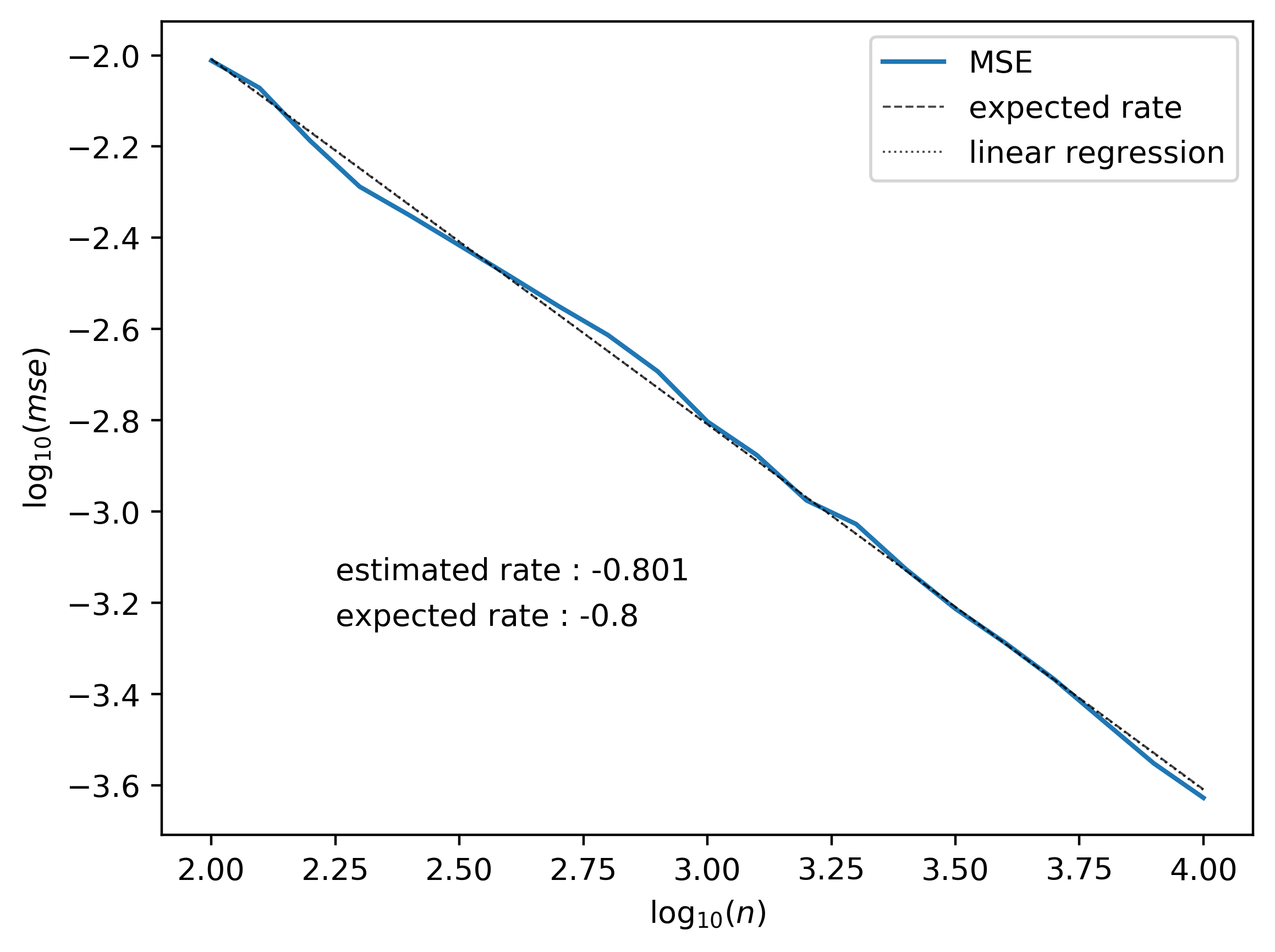}
\includegraphics[scale=0.5]{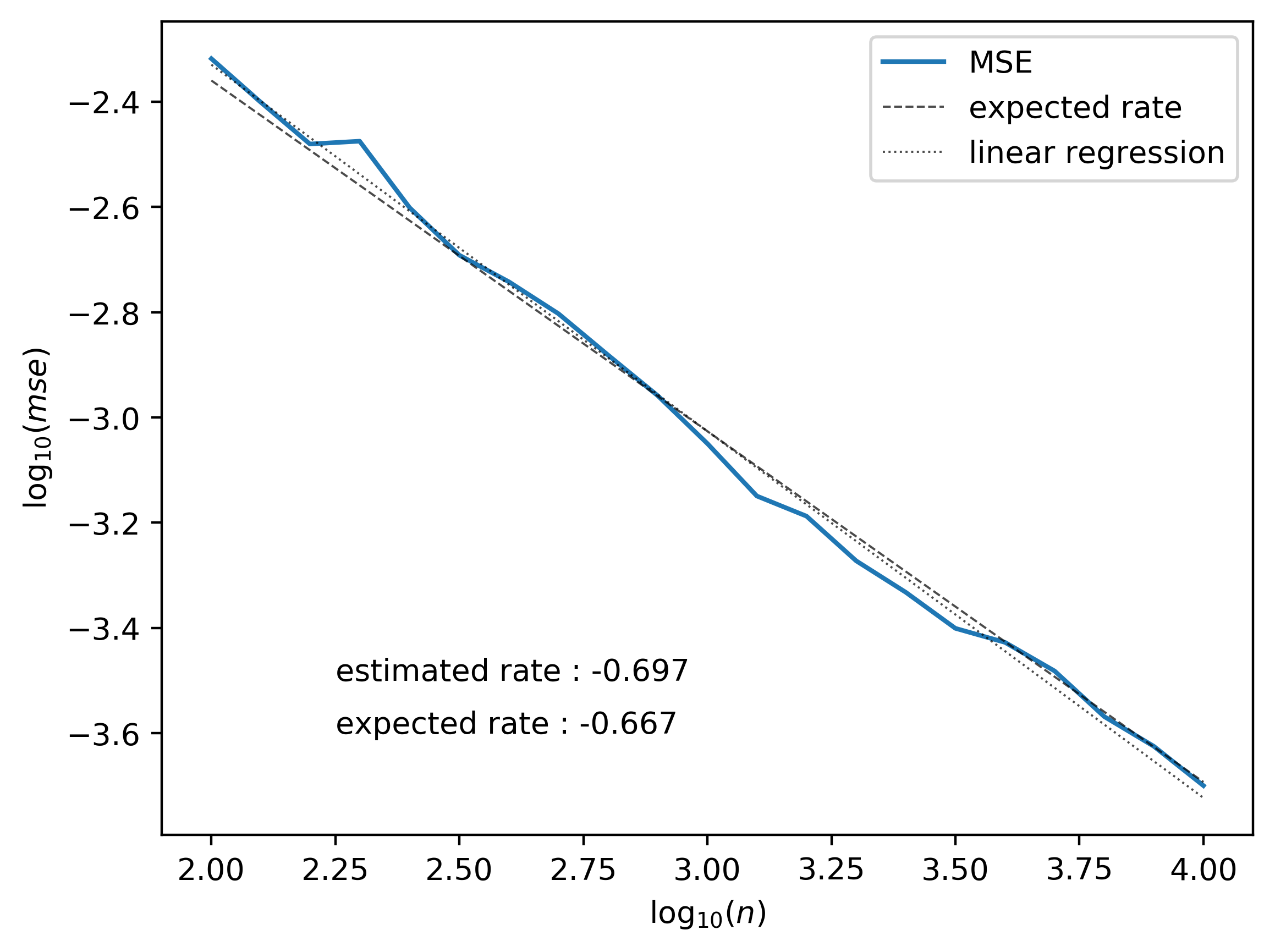}
\captionsetup{justification = justified, margin = 1.1cm}
\caption{\small Plot of the empirical mean square error (blue) for a density supported by a one-dimensional submanifold (Left) and two-dimensional submanifold (Right) with parameter $\beta = 2$. We use a log-regular grid $\bn$ of $21$ points ranging from $100$ to $10^4$. Each experiment is repeated $N = 500$ times. }
\label{fig:msed1}
\end{figure}

\subsection{Adaptation}

In this section we estimate a density when its regularity is unknown, contrary to the previous simulation where the regularity parameter $\beta$ is pugged in the bandwidth choice $n^{-1/(2\beta+d)}$. This is performed using Lepski's method presented in \secref{smooth}. The rate is computed using \algoref{main}, for both the one-dimensional and the two-dimensional synthetic datasets.

For the adaptive estimation on the two-dimensional manifold, we observe that the corrective term $\det d\Psi(0)$ computed in \eqref{detphi} results in a density $F_\beta(x)$ that is quite small, while the function $\psi$ defined at \eqref{defpsi} and used to tune the bandwidth soars dramatically because of the retained value of $\omega_d = 4^d \zeta_d \|K^{(d,\ell)}\|^2_\infty f_{\max}$, so that the values of $\hat f_h$ and $\psi(h, \cdot)$ (defined at \eqref{defpsi}) are not of the same order anymore at this scale (using maximum $10^6$ observations). To circumvent this effect, we introduce a scaling parameter $\lambda$ as follows 
$$
G_{\beta,\lambda} = v,u \mapsto \lambda^2 G_\beta(\lambda v, \lambda u).
$$ 
Like before, we consider the push-forward probability measure $\Psi_*G_{\beta,\lambda}\cdot\mu_{\bbT^2}$ which has density $F_{\beta,\lambda}$ with respect to $\mu_M$. For $\lambda = 4$, we find that $F_{\beta,\lambda}(x)$ is of order $1$ for most values of $\beta$, and we use the function $\psi^{\text{num}}(h,\eta) = \Omega^{\text{num}}(h)\lambda(h) + \Omega^{\text{num}}(\eta)\lambda(\eta)  $ using simply $\Omega^{\text{num}}(h) = \sqrt{1/nh^d}$. We have no theoretical guarantee that such a method will work but we recover nonetheless the right rate in the estimation of the value of the density, see \figref{adapt} for a plot of the estimated rate. 

We find a highly dispersive empirical error, hence our choice to represent the median of the squared error instead of the more traditional  mean squared error. 

\begin{figure}[ht!]
\centering
\includegraphics[scale=0.5]{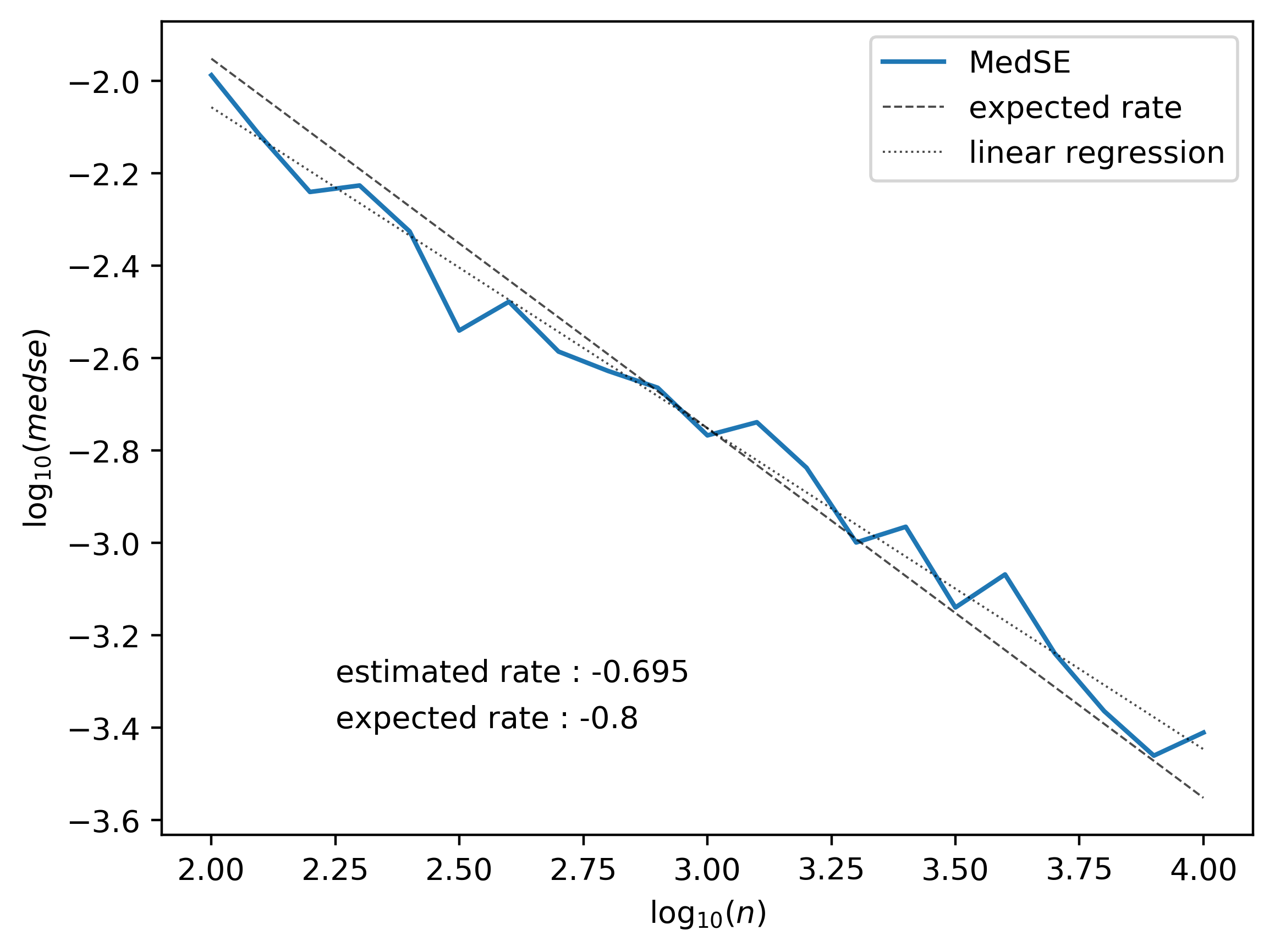}
\includegraphics[scale=0.5]{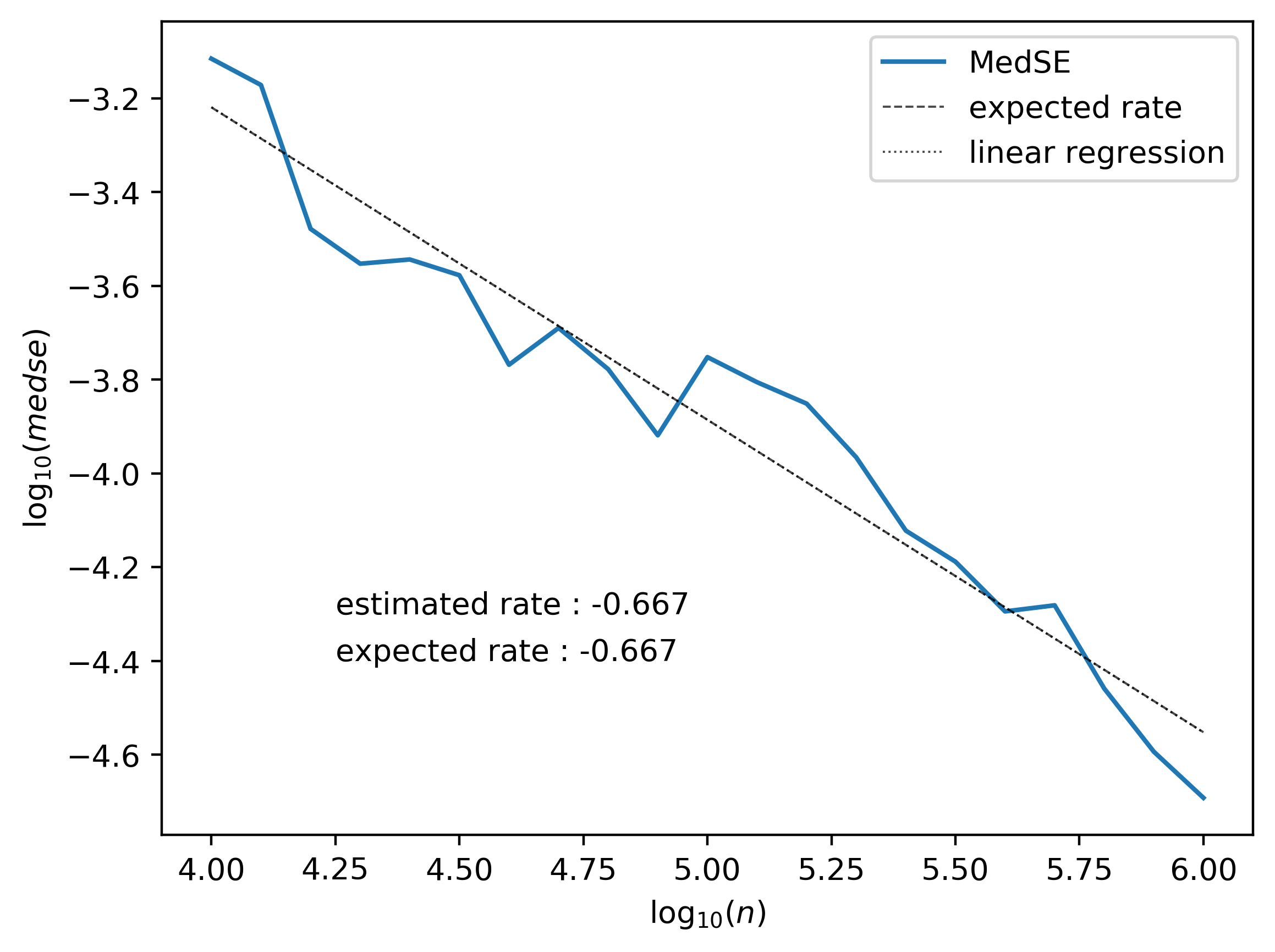}
\captionsetup{justification = justified, margin = 1cm}
\caption{\small (Left) Plot of the empirical median square error for the one-dimensional submanifold with $\beta = 2$. The bandwidth $h$ is chosen adaptively using Lepski's method of order $\ell = 3$ as in \secref{smooth}. We used a log-regular grid $\bf n$ of $11$ points ranging from $100$ to $10^4$ and each experience was repeated $N = 500$ times. (Right) Same experiment but for the two dimensional manifold, with a grid ranging from $10^4$ to $10^6$ and $N = 100$ repetitions.} 
\label{fig:adapt}
\end{figure}

 \noindent {\bf Acknowledgements} {\it We are grateful to Krishnan (Ravi) Shankar and Hippolyte Verdier for insightful discussions and comments. The valuable input of two referees is greatly acknowledged.}

\bibliographystyle{chicago}
\nocite{*}
\bibliography{ref}

 \renewcommand{\thesection}{A}
 \section{Appendix} \label{proofs}
 \addcontentsline{toc}{section}{Appendices}

\subsection{Additional results of geometry} \label{app:geo}

We first state a few classical results that we will need in the upcoming proofs. We start with a quantitative bound that link the reach to the curvature of a submanifold. We denote by $\II$ the second fundamental form.

\begin{prp} \emph{\citep[Prp. 6.1]{NSW08}}  \label{prop NSW}
Let $M$ be a compact smooth submanifold of $\bbR^D$. Then, for any $z \in M$, we have $\|\II_z\|_{\op} \leq 1/\tau_M$.
\end{prp}

Since $\II_z$ is the differential of order two of the mapping $\exp_z$ at the $0 \in T_x M$, Proposition \ref{prop NSW} has several convenient implications. First, it gives a uniform lower bound for the injectivity radii of $M$ as stated in \prpref{injexp}. Second, it also yields nice bounds on how well the Euclidean distance on $\bbR^D$ approximates the Riemannian distance $d_M$ on $M \times M$.

 \begin{prp} \label{prp:eqdist} \emph{\citep[Prp. 6.3]{NSW08}} For any compact submanifold $M$ of $\bbR^D$ and any $x,y \in M$ such that $\|x-y\| \leq \tau_M/2$, we have
 \beq 
 \|x-y\| \leq d_M(x,y) \leq \tau_M \(1 - \sqrt{1 - \frac{2 \|x-y\|}{\tau_M}}\). \nonumber
 \eeq
 \end{prp}
 
Proposition \ref{prp:eqdist} allows in turn to compare the volume measure $\mu_M$ to the Lebesgue measure on its tangent spaces. 
 
 \begin{lem} \label{lem:abstand} For any $d$-dimensional compact smooth submanifold $M$ of $\bbR^D$, for any $z \in M$ and any $\eta \leq \tau_M/2$, we have
 \beq \nonumber
(1 - \eta^2/6\tau_M^2)^d \zeta_d \eta^d \leq \mu_M(B(z,\eta)) \leq \{(1+(\xi(\eta/\tau_M )\eta)^2/\tau_M^2) \xi(\eta/\tau_M)\}^d \zeta_d \eta^d
 \eeq
where $\xi(s) = (1 - \sqrt{1 - 2 s})/s$ and $\zeta_d$ is the volume of the unit Euclidean ball in $\bbR^d$. 
 \end{lem}
 
 \begin{proof} This result already appears in \citep[Lem III.23]{Aam17} but we prove it here to make constants explicit. Let us denote by $\Leb$ the Lebesgue measure on $T_x M$. Using \citep[Prp III.22.v]{Aam17}, we know that, as long as $\xi(\eta/\tau_M)\eta \leq \tau_M$ (which holds if $\eta \leq \tau_M / 2$), 
 \begin{align*}
(1 - \eta^2/6\tau^2)^d \Leb(B_{T_z M}(0, \eta)) &\leq \mu_M\big(\exp_z(B_{T_z M}(0,\eta))\big) \\  &\leq  \mu_M\big(\exp_z (B_{T_z M}(0,\xi(\eta/\tau_M)\eta))\big) \\
&\leq (1 + (\xi(\eta/\tau_M)\eta)^2/\tau_M^2)^d \Leb(B_{T_z M}(0, \xi(\eta/\tau_M)\eta)).
 \end{align*}
Thanks to \prpref{eqdist}, if $\eta \leq \tau_M/2$, then $\exp_z\big( B_{T_z M}(0,\eta)\big) \subset M \cap B(x,\eta) \subset \exp_{z} \big(B_{T_z M}(0,\xi(\eta/\tau_M)\eta)\big)$. These inclusions combined with the last inequalities yield the result.
 \end{proof}

\subsection{Proof of \thmref{reachrisk}} \label{app:loss}

We go along a classical line of arguments, thanks to a Bayesian two-point inequality by means of Le Cam's lemma \citep[Lem. 1]{Yu97}, restated here in our context. For two probability measures $P_1,P_2$, we write $\tv(P_1,P_2) = \sup_{A}|P_1(A)-P_2(A)|$ for their variational distance and $H^2(P_1, P_2) = \int \(\sqrt{dP_1}-\sqrt{dP_2}\)^2$ for their (squared) Hellinger distance.
\begin{lem}{\emph{(Le Cam)}}  \label{lem:lecam} 
	For any $P_1, P_2 \in \sdab$, we have,
	\begin{align}
	\inf_{\widehat f} \sup_{P \in \sdab} \bbE_{P^{\otimes n}}[|\wh f(x) - f_P(x)|^p]^{1/p} &\geq \frac12 |f_{P_1}(x) - f_{P_2}(x)| \(1 - \tv\(P_1^{\otimes n}, P_2^{\otimes n}\)\) \label{yu} \\
	&\geq \frac12 |f_{P_1}(x) - f_{P_2}(x)| \(1 - \sqrt{2 - 2(1 - H^2(P_1,P_2)/2)^n}\). \nonumber
	\end{align}
\end{lem}

\begin{proof} The proof of \eqref{yu} can be found in \citet[Lem. 1]{Yu97}.  It only remains to see that $\tv\(P_1^{\otimes n}, P_2^{\otimes n}\) \leq \sqrt{2 - 2(1 - H^2(P_1,P_2)/2)^n}$.  This comes from classical inequalities on the Hellinger distance, see \citet[Lem. 2.3 p.86]{Tsy08} and  \citet[Prp.(i)-(iv) p.83]{Tsy08}.
\end{proof}

\begin{proof}[Proof of \thmref{reachrisk}] With no loss of generality, we pick $x = 0$. We work in $\bbR^{d+1} \subset \bbR^D$, and denote $(e_1,\dots,e_{d+1})$ the canonical basis of $\bbR^{d+1}$. We consider a family of submanifolds $M_\delta \subset \bbR^{d+1}$ such that
$$
M_\delta~ \bigcap ~\{(z,t) \in \bbR^d \times \bbR~|~\| z\| \leq 1\} = O(\delta) \cup O(-\delta) 
$$
where  $O(t) = \{(z,t)~|~z \in \bbR^d, \| z\| \leq 1\}$. We do not give the construction explicitly, but refer instead to \figref{md} for a diagram of such a manifold. 

\begin{figure}[ht!]
\centering
\includegraphics[scale=.3]{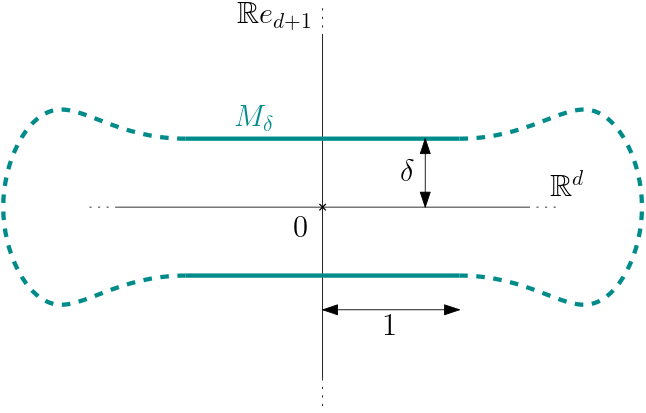}
\captionsetup{justification =  justified, margin = 1.1cm}
\caption{\small Diagram of a candidate for $M_\delta$.}
\label{fig:md}
\end{figure}

We endow each $M_\delta$ with a density $f_\delta$ such that
$$
\forall z \in O(\delta),~~f_\delta(z) = f_{\max}~~~\text{and}~~~\forall z \in O(-\delta),~~f_\delta(z) = f_{\min}
$$
and we denote $Q_{\delta} = f_\delta d\mu_{M_\delta}$. If  $f_{\min}$ is small enough (due to the constraint $\vol \supp P \leq 1/f_{\min}$ for any $P \in \sdab$) we can always choose $M_\delta$ and $f_\delta$ so that $\|\iota_{M_{\delta}}\|_{\cH^{\alpha+1}} \leq R/2$ and $\|f_\delta\|_{\cH^\beta} \leq L/2$. 

Let now $\Phi : \bbR^{d+1} \to \bbR$ be a smooth, positive, radial function with support in $B(0,1)$ with $\Phi(0) = 1$. Because the exponential map smoothly depends on the metric, for any $h < 1$, there exists $\delta_h \in (0,h)$ sufficiently small such that the push-forward measures of $Q_{\delta_h}$ through the mappings
\begin{align*} 
\Psi^+_h(z) = \Id - \delta_h \Phi\(\frac{z - \delta_h e_{d+1}}{h}\) e_{d+1} ~~~~\text{and}~~~~\Psi^-_h(z) = \Id + \delta_h \Phi\(\frac{z + \delta_h e_{d+1}}{h}\) e_{d+1}
\end{align*} 
are both in $\sdab$. We write $N_h^\pm = \Psi^\pm_h(M_{\delta_h})$, $P_h^\pm = \(\Psi^\pm_h\)_* (Q_{\delta_h})$ and $g_h^\pm$ for the continuous version of the density $dP_h^\pm/d\mu_{N_h^\pm}$. See \figref{nh} for a diagram of $N_h^+$ and $N_h^-$. 

\begin{figure}[ht!]
\centering
\includegraphics[scale=.3]{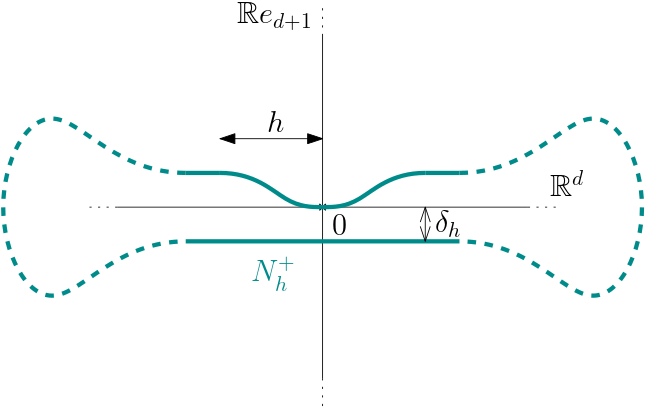} \hspace{1cm}
\includegraphics[scale=.3]{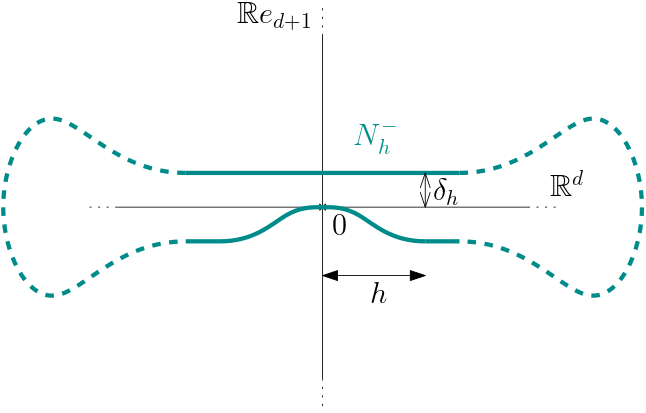}
\captionsetup{justification =  justified, margin = 1.1cm}
\caption{\small Diagram of manifolds $N_h^+$ (Left) and $N_h^-$ (right).}
\label{fig:nh}
\end{figure}

Using \lemref{lecam}, we obtain
$$
\inf_{\widehat f} \sup_{P \in \sdab} \bbE_{P^{\otimes n}}[|\wh f(0) - f_P(0)|^p]^{1/p} \geq \frac12 |g_h^+(0) - g_h^-(0)| \(1 - n \tv\(P_h^{+}, P_h^{-}\)\). 
$$
But now $g_h^+(0) = f_{\delta_h} (\delta_h e_{d+1}) \times | \det d\Psi_h^+(\delta_h e_{d+1})|^{-1} = f_{\max}$ and, likewise, $g_h^-(0) = f_{\min}$. As for the total variation distance, we get
\begin{align*} 
\tv\(P_h^{+}, P_h^{-}\) &= P_h^+(\Psi_h^+(h O(\delta_h)) + P_h^+(h O(-\delta_h))  + P_h^-(\Psi_h^-(hO(-\delta_h)) + P_h^-(hO(\delta_h)) \\
 &= 2 \vol (h O(\delta_h)) \times \(f_{\min} + f_{\max}\) = 2 \zeta_d h^d \(f_{\min} + f_{\max}\) 
\end{align*}
where we recall that $\zeta_d$ is the volume of the $d$-dimensional unit-ball. Putting all the estimates together, we conclude
$$
\inf_{\widehat f} \sup_{P \in \sdab} \bbE_{P^{\otimes n}}[|\wh f(0) - f_P(0)|^p]^{1/p} \geq \frac12 (f_{\max} - f_{\min}) \(1 - 2 n  \zeta_d h^d \(f_{\min} + f_{\max}\)\).
$$ 
Letting $h$ goes to $0$  yields the result.
\end{proof}

\subsection{Proof of \thmref{lower}} \label{app:prooflower}

\begin{proof}[Proof of \thmref{lower}] Suppose without loss of generality that $x = 0$ and consider a smooth submanifold $M$ of $\bbR^{d+1} \subset \bbR^D$ that contains the disk $B_{\bbR^d}(0,1) \subset \bbR^d \times \{0_{\bbR^{D-d}}\}$ with reach greater than $\tau$, see \figref{lbound} for a diagram of such an $M$. By smoothness and compacity of $M$, there exists $L_*$ (depending on $\tau$) such that $M \in \cC_{d,\alpha}(\tau,L_*)$. Let $P$ be the uniform probability measure over $M$, with density $f : x \mapsto1/\vol M$. We have $P \in \sdab$ as long as $L^* \leq L$ and $f_{\min} \leq 1/\vol M \leq f_{\max}$ an assumption we make from now on. For $0 < \delta \leq 1$, let  $P_\delta = f_\delta \cdot d \mu_M$ with
$$
f_\delta(y) = \begin{cases} f(y) + \delta^{\beta} G(y/\delta)~~~\text{if}~~y\in B(0,\delta) \\
f(y)~~~\text{otherwise}
\end{cases}
$$
with $G : \bbR^d \rightarrow \bbR$ a smooth function with support in $B_{\bbR^d}(0,1)$ and such that $\int_{\bbR^d} G(y)dy = 0$. We pick $G$ such that $f_\delta \in \cF_\beta$ for small enough $\delta$, depending on $\tau$. Such a $G$ can be chosen to depend on $R$ only.

\begin{figure}[ht!]
\centering
\includegraphics[scale=.12]{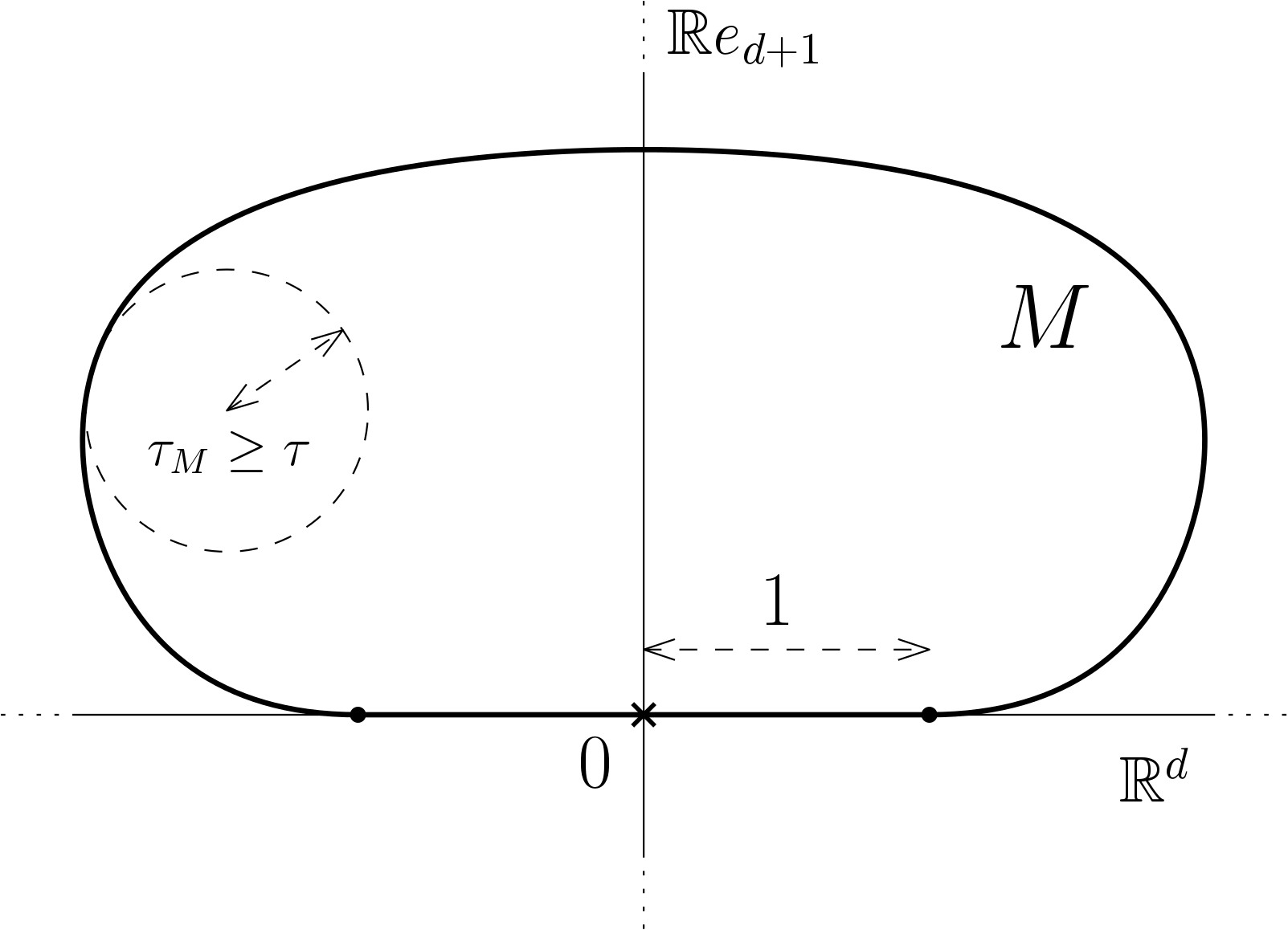}
\includegraphics[scale=.12]{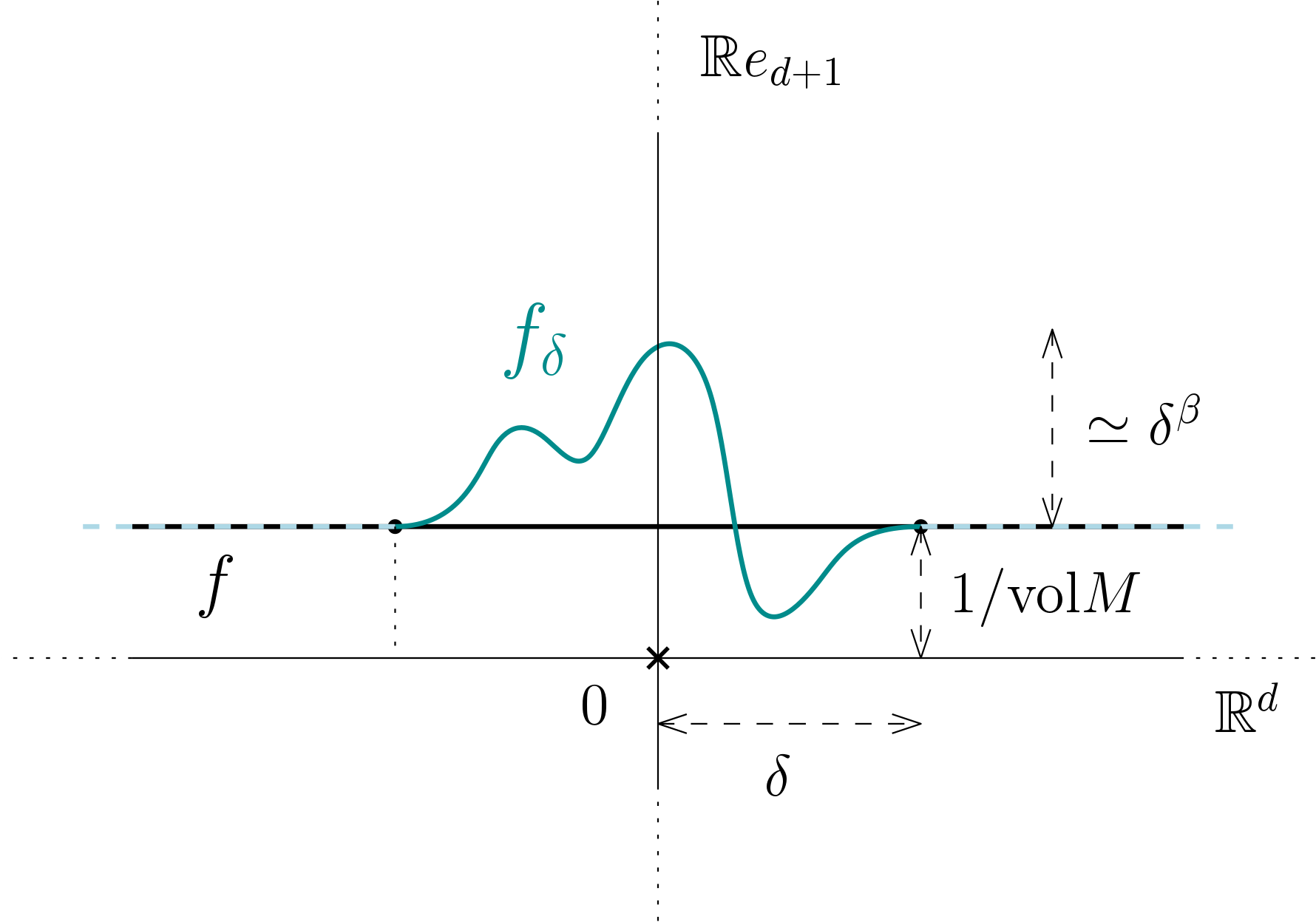}
\captionsetup{justification =  justified, margin = 1.1cm}
\caption{\small Diagram of a candidate for $M$ (Left) and of the densities $f$ and $f_\delta$ around $0$ (Right).}
\label{fig:lbound}
\end{figure}

For $\delta$ small enough (depending on $\tau$), we thus have $P_\delta \in \sdab$ as well. By \lemref{lecam}, we infer
\beq
\inf_{\widehat f} \sup_{P \in \sdab} \bbE_{P^{\otimes n}}[|\wh f(x) - f_P(x)|^p]^{1/p}  \geq \tfrac12 \delta^\beta |G(0)| \(1 - \sqrt{2 - 2(1 - H^2(P,P_\delta))^n}\) \nonumber
\eeq
so that it remains to compute $H^2(P,P_\delta)$. We have the following bound
\begin{align*}
H^2(P,P_\delta) &= \int_{B_{\bbR^d}(0,\delta)} (1 - \sqrt{1 + \vol M \delta^\beta G(z /\delta)})^2 dz \\
&\leq  \int_{B_{\bbR^d}(0,\delta)}  (\vol M)^2 \delta^{2\beta} G^2(z /\delta)dz \\
&\leq (C \vee 1) \delta^{2\beta+d}
\end{align*}
with $C  = \vol M \times \int_{B(0,1)} G(z)^2 dz$ depending on $\tau$ and $R$ only. Taking $\delta = (1/(C\vee 1) n)^{1/(2\beta+d)}$ we obtain, for large enough $n$ (depending on $\tau$)
\begin{align*}
\inf_{\widehat f} \sup_{P \in \sdab} \bbE_{P^{\otimes n}}[|\wh f(x) - f_P(x)|^p]^{1/p}   \geq \tfrac12 \big((C \vee 1) n\big)^{-\beta/(2\beta+d)} \sqrt{2 - 2(1 - 1/n)^n} \geq C_* n^{-\beta/(2\beta+d)},
\end{align*}
with $C_* = (C \vee 1)^{-1/2}$ depending on $\tau$ and $R$.
\end{proof}

\subsection{Proofs of \secref{ker}} \label{app:ker}

We set $K_h(z) = h^{-d}K(z/h)$ and start with bounding the variance of $K_h(X-x)$ when $X$ is distributed according to $P \in \sdab$. Let first observe that 
\beq
|K_h(X-x)| \leq \frac{\|K\|_\infty}{h^d} \ind_{B_D(x,h)}(X) \leq \frac{\|K\|_\infty}{h^d} \label{c}
\eeq

\begin{lem} \label{lem:var} For any $P \in \sdab$ and for any $h  \leq \tau/2$,
\beq
\Var_P(K_h(X-x)) \leq \frac{\omega}{h^d}~~~\text{where}~~\omega = 4^d \zeta_d \|K\|_\infty^2 b. \nonumber
\eeq
with $\zeta_d$ being the volume of the unit ball in $\bbR^d$. 
\end{lem}

\begin{proof} We have
\begin{align}
\Var_P(K_h(X-x)) \leq \bbE_P[ K_h(X-x)^2 ] \leq \frac{\|K\|_\infty^2}{h^{2d}} P(B(x,h)) \leq  \frac{4^d \zeta_d f_{\max} \|K\|_\infty^2}{h^d} \nonumber
\end{align}
where we used \eqref{c} and \lemref{abstand} with $\eta = \tau/2$. 
\end{proof}

Using Bernstein inequality \citep[Thm. 2.10 p.37]{BLM13}, for any $P \in \sdab$ and any $t > 0$,  we infer 
\beq \label{bern}
\mathbb P\(\big|\widehat  \xi_h(P,x)\big| \geq \sqrt{\frac{2\omega t}{nh^d}} + \frac{\|K\|_\infty t}{nh^d} \)  \leq 2e^{-t}, 
\eeq
where $\mathbb P$ is a short-hand notation for the distribution $P^{\otimes n}$ of the $n$-sample $X_1,\ldots, X_n$ taken under $P$. The bound \eqref{bern} is the main ingredient needed to bound the $L_p$-norm of the stochastic deviation of $\widehat f_h$.

\begin{proof}[Proof of \prpref{sd}] We denote by $u_+ = \max\{u,0\}$ the positive part of a real number $u$. We start with 
\beq
\bbE_{P^{\otimes n}}\big[ |\widehat  \xi_h(P,x)|^p \big]\leq 2^{p-1} \(\Omega(h)^p + \bbE_{P^{\otimes n}}\big[\big(|\widehat  \xi_h(P,x)|-\Omega(h)\big)^p_+\big]\). \nonumber
\eeq
The first term has the right order. For the second one, we make use of \eqref{bern} to infer
\begin{align*}
\bbE_{P^{\otimes n}} \big[\big(|\widehat  \xi_h(P,x)|-\Omega(h)\big)^p_+\big] &=  \int_0^\infty \bbP\(|\widehat  \xi_h(P,x)| > \Omega(h) + u\)  pu^{p-1} du \\
&= p\Omega(h)^p \int_0^\infty \bbP\( |\widehat  \xi_h(P,x)| > \Omega(h)(1 + u) \)  u^{p-1} du \\
&\leq  p\Omega(h)^p \( 1 + \int_1^\infty \bbP\( |\widehat  \xi_h(P,x)| > \sqrt{\frac{2\omega(1+u)}{nh^d}} + \frac{\|K\|_\infty (1+u)}{nh^d} \)  u^{p-1} du \) \\
&\leq p\Omega(h)^p \( 1 + \int_1^\infty 2 e^{-1-u} u^{p-1} du \) \\
& \leq p\Omega(h)^p (1 + \Gamma(p))
\end{align*}
which ends the proof. 
\end{proof}

The proof of \lemref{expans} partly relies on the following elementary lemma.

\begin{lem} \label{lem:tech}
Let $\gamma \geq 0$ be a real number and let $g : \bbR^m \rightarrow \bbR$ for $m \in \bbN^*$ satisfying that $\|g\|_\infty \leq b$ and that the restriction of $g$ to $B(0,r)$ (denoting here the open ball in $\bbR^m$) is $\beta$-H\"older, meaning that
\beq
\forall v,w \in B(0,r),~~~~\|d^k g(v) - d^k g(w)\| \leq A \|v -w\|^\delta \nonumber
\eeq
for some $A > 0$ with $k = \ceil{\gamma -1}$ and $\delta = \gamma - k$. Then there exists a constant $C$ (depending on $m,\gamma,r, b$ and $A$, and depending on $m$ and $\gamma$ when $r = \infty$) such that, for all $1 \leq j \leq k$,
\beq
\sup_{v \in B(0,r/2)} \|d^j g(v)\|_{\op} \leq C b^{1-j/\gamma} A^{j/\gamma}. \nonumber
\eeq
\end{lem} 
\begin{proof} 
Let $v \in B(0,r/2)$. Since $g$ is $\gamma$-H\"older on $B(0,r)$, we know that there exists a function $R_v$ such that, for any $z$ such that $v+z \in B(0,r)$, we have
\beq
g(v + z) - \sum_{j=0}^{k} \frac{1}{j!}d^j g(v)[z^{\otimes k}] = R_v(z) \nonumber
\eeq
with $|R_v(z)| \leq A \|z\|^{\beta} / k!$. Let $h = (2b{k}!/A)^{1/\gamma}$, and $z_0 \in \bbR^m$ be unit-norm. Pick $a_1,\dots,a_k \in (0,1)$ all distincts and small enough such that $h a_k z_0 \in B(0,r/2)$ for all $k$ (if $r = \infty$, then we can pick the $a_i$ independently from $A$, $b$ and $\gamma$). Introducing the vectors of $\bbR^k$
\begin{align*} 
X &= \(h dg(v)[z_0],\dots, \frac{h^{k}}{k!} d^{k} g(v)[z^{\otimes k}_0] \) ~~~~\text{and} \\
Y &= \(g(v + ha_1z_0) - g(v) - R_v(ha_1z_0),\dots, g(v + ha_k z_0) - g(v) - R_v(ha_k z_0)\)
\end{align*}
we have $Y = VX$ with $V$ being the Vandermonde matrix associated with the real numbers $(a_1,\dots,a_k)$. The former being invertible, we have $\|X\| \leq \|V^{-1}\|_{\op} \|Y\|$ and thus, for any $1 \leq j \leq k$
\beq
\left|\frac{h^j}{j!} d^j g(v)[z_0^{\otimes k}]\right| \leq \|V^{-1}\|_{\op} \( 2b + \frac{A}{k !}  h^\gamma \). \nonumber
\eeq
Substituing the value of $h$ and noticing that the former inequality holds for every unit-norm vector $z_0$, we can conclude.

\end{proof}

\begin{proof}[Proof of \lemref{expans}] 

We set $B_h = B(x,h)$. Since $\tau/2$ is smaller than the injectivity radius of $\exp_x$ (see \prpref{injexp}) we can write
\beq \label{int1}
f_h(P,x) =\int_{B_h} K_h\(p-x\)f(p) d\mu_M(p) = \int_{\exp_x^{-1}B_h} K_h(\exp_x v - x) f(\exp_x v) \zeta(v) dv
\eeq
with $\zeta(v) = \sqrt{\det g^x(v)}$. We set $\gamma = \alpha \wedge \beta$ and $k = \ceil{\gamma-1}$. Let $F$ denote the map $f \circ \exp_x$. For $h$ smaller than $\tau/2$, we have $\exp_x^{-1} B_h \subset B_{T_x M}(0,2h) \subset B_{T_x M}(0,\tau)$ (see \prpref{eqdist}). We can thus write the following expansion, valid for all $v \in \exp_x^{-1} B_h$ and all $w \in T_x M$,
\begin{align} \label{dev1}
\exp_x(v) = x + v + \sum_{j=2}^{k+1} \frac{1}{j!} d^j \exp_x(0)[v^{\otimes j}] + R_1(v)~~~&\text{with}~~\|R_1(v)\| \leq C_1 \|v\|^{\gamma+1}, \\
F(v) = f(x) + \sum_{j=1}^{k}  \frac{1}{j!} d^j F(0)[v^{\otimes j}] + R_2(v)~~~&\text{with}~~|R_2(v)| \leq C_2 \|v\|^{\gamma},  \\
K(v + w) =  \ind_{\{\|v+w\|\leq 1\}} \( K(v) +  \sum_{j=1}^{k}  \frac{1}{j!} d^j K(v)[w^{\otimes j}] + R_3(v,w) \)~~~&\text{with}~~|R_3(v,w)| \leq C_3 \|w\|^\gamma,
\end{align}
with $C_1$ depending on $\alpha, \tau$ and $L$, $C_2$ depending on $\beta,\tau,f_{\max}$ and $R$ (see \lemref{tech}), and $C_3$ depending on $K$. Since now we know that $g^x_{ij}(v) = \<d\exp_x(v)[e_i], d\exp_x(v)[e_j] \>$, we have a similar expansion for the mapping $\zeta(v) = \sqrt{\det g^x(v)}$ 
\begin{align}
\zeta(v) = 1 + \sum_{j=1}^{k}  \frac{1}{j!} d^j \zeta(0)[v^{\otimes j}] + R_4(v)~~~&\text{with}~~|R_4(v)| \leq C_4 \|v\|^{\gamma}  \label{dev4}
\end{align}
with $C_4$ depending on $\alpha, \tau$ and $L$. Making the change of variable $v = hw$ in \eqref{int1}, we get 
\begin{align*}
f_h(P,x) = \sum_{k=0}^{k} G_j(h,P,x) + R_h(P,x)
\end{align*}
with $G_j$ corresponding to the integration of the $j$-th order terms in the expansion around $0$ of the function  $v \mapsto K\(\frac{p-\exp_p (hw)}{h}\)F(hw) \zeta(hw)$. In particular $G_j$ can be written as a sum of terms of the type 
\beq
I = h^j \int_{\frac{1}{h}\exp_x^{-1} B_h} d^m K(w)[\phi(w)^{\otimes m}] \psi(w) dw \nonumber
\eeq
where $\psi$ and $\phi$ are monomials in $w$ satisfying $m \deg \phi + \deg \psi = j$, with coefficients  bounded by constants depending on $\alpha,  \tau, L, \beta, f_{\max}$ and $R$ (again, use \lemref{tech} to bound the derivatives). Since now $B_{T_x M}(0,1) \subset \frac{1}{h} \exp_x^{-1} B_h$, and since $d^j K$ is zero outside of $B(0,1)$, we have that $G_j$ can actually be written $G_j(h,P,x) = h^j G_j(P,x)$ with $|G_j(P,x)| \leq C$ for some $C$ depending on $K, \alpha,  \tau, L, \beta, f_{\max}$ and $R$. Similar reasoning leads to $R_h(P,x) \leq C h^\gamma$ with $C$ depending again on  $K, \alpha,  \tau, L, \beta, f_{\max}$ and $R$. To conclude, it remains to compute $G_0(P,x)$. Looking at the zero-th order terms in the expansions \eqref{dev1} to \eqref{dev4}, we find that
\beq
G_0(P,x) = \int_{B_{T_x M}(0,1)} K(w) f(x) dw = f(x) \nonumber
\eeq
where we used \assref{kernel}. The proof of \lemref{expans} is complete. 
\end{proof}

\begin{proof}[Proof of \prpref{bias}] 
For a positive integer $\ell \geq 1$, let $f_h^{(\ell)}(P,x)$ be the mean of the estimator $\widehat f_h(x)$ computed using $K^{(\ell)}$. Let $\gamma = \alpha\wedge \beta$ and $k = \ceil{\gamma - 1}$. We recursively prove on $1 \leq \ell < \infty$ the following identity
\beq \label{recurs}
\forall h \leq \tau/2,~~~f^{(\ell)}_h(P,x) = f(x) + \sum_{j =\ell}^{k}  h^j G^{(\ell)}_j(P,x) + R^{(\ell)}_h(P,x)
\eeq
where $|R^{(\ell)}_h(P,x)| \leq C^{(\ell)} h^{\gamma}$ for some constant $C^{(\ell)}$ depending on $\tau, \ell, L, R, f_{\max}$ and $\beta$. The initialisation step $\ell = 1$ has been proven in \lemref{expans}. Let now $1\leq \ell \leq k$. By linearity of $f_h(P,x)$ with respect to $K$, we have
\beq
f^{(\ell+1)}_h(P,x) = 2 f^{(\ell)}_{2^{-1/\ell}h}(P,x) - f^{(\ell)}_h(P,x). \nonumber
\eeq
Since $2^{-1/\ell} h \leq h$, we can use our induction hypothesis \eqref{recurs} and find
\beq
f^{(\ell+1)}_h(P,x) = f(x) + \sum_{j =\ell}^{k} (2^{1-j/\ell} - 1) h^j G^{(\ell)}_j(P,x) + 2R^{(\ell)}_{2^{-1/\ell}h}(P,x) - R^{(\ell)}_h(P,x). \nonumber
\eeq
We conclude noticing that $2^{1-j/\ell} - 1 = 0$ for $j = \ell$, and setting $G^{(\ell+1)}_j(P,x) = (2^{1-j/\ell} - 1 )G^{(\ell)}_j(P,x)$ and $R^{(\ell+1)}_h(P,x) = 2R^{(\ell)}_{2^{-1/\ell}h}(P,x) - R^{(\ell)}_h(P,x)$. The new remainder term verifies
\beq
|R^{(\ell+1)}_h(P,x) | \leq (2^{1 -\gamma/\ell} + 1) C^{(\ell)} h^\gamma \leq 3C^{(\ell)} h^\gamma
\eeq 
ending the induction by setting $C^{(\ell + 1)} = 3C^{(\ell)}$. When $\ell \geq k+1$, the induction step is trivial.
\end{proof}

\subsection{Proofs of \secref{uni}} \label{app:uni}

\begin{proof}[Proof of \lemref{dim1}]   Let $\gamma : [0,L_M] \rightarrow M$ be a unit speed parametrization of $M$ and extend $\gamma$ to a smooth function on $\bbR$ by $L_M$-periodicity. Suppose without loss of generality that $\gamma(0) = z$. For any $t \in \bbR$, there is a canonical identification between $T_{\gamma(t)} M$ and $\bbR$ through the map $v \mapsto \inner{\dot \gamma(t)}{v}$. With such an identification, we can write that for $s \in \bbR \simeq T_z M$, $\exp_z(s) = \gamma(s)$ because $\gamma$ is unit-speed. We thus have $d\exp_z(s)[h] = h\dot \gamma(s)$ for any $h \in T_{\gamma(s)} M \simeq \bbR$. It follows that  $\det g^z(s) = \|d\exp_z(s)[1]\|^2 = \|\dot \gamma(s)\|^2 = 1$ and this completes the proof. 
\end{proof}

We write $V = \(x, X_1,\dots,X_n\) $ for the vertices of $\cG_\ve$ and $\widehat \eta = \sup_{z \in M} d(z,V)$. For small enough $\widehat  \eta$ we have that $\cG_\ve$ is connected, therefore the distance $\widehat d_\ve$ is well-defined on $V$. We have in that case a good reverse control of $d_M$ by $\wh d_\ve$, as shown in the next two lemmata.

\begin{lem} \label{lem:dd1} If $\ve \leq 8\tau$ and $16\widehat \eta \leq \ve$, then $\widehat  d_\ve(y,z) \leq d_M(y,z)$ for any $y,z \in V$. 
\end{lem}

\begin{lem} \label{lem:dd2} If $\ve \leq \tau/ 2$, then $d_M(y,z) \leq  \(1 + \frac{\pi^2}{48\tau^2} \ve^2\) \wh d_\ve(p,q)$ for any $y,z \in V$. 
\end{lem}

\begin{proof}[Proof of \lemref{dd1}]  We can take the shortest path in $M$ between $y$ and $z$ as a unit-speed path of the form $\gamma : [0,\ell] \rightarrow \bbR^D$ with $\ell = d_M(y,z) \leq L_M/2$. We let $\delta = \ell/(4\floor{\ell/\ve})$ and $N = 4\floor{\ell/\ve}$. Notice that $\ve/4 \leq \delta \leq \ve/2$.  Let us define $p_j = \gamma(j \delta)$, so that $p_0 = y$ and $p_N = z$. Since $\widehat  \eta \leq \ve/16$, for every $1 \leq j \leq N-1$, there exists among our vertices $V$ a point denoted by $\widehat  p_j$  such that $\|p_j - \widehat  p_j\| \leq \ve / 16$. We set $\widehat  t_j \in [0, L_M]$ for its coordinate, namely $\widehat p_j = \gamma(\widehat t_j)$. 

Let us show first that for $1 \leq j < N$, we have $\widehat t_j \in [0,\ell]$. Indeed, thanks to \prpref{eqdist}, since $\ve/16 \leq \tau/2$, we have $|t_j - \wh t_j| \leq 2 \|p_j - \wh p_j\| \leq \ve/8$. Since $\delta \geq \ve/4$, we thus have $0 \leq \widehat  t_1 \leq \dots \leq \widehat  t_{N-1} \leq \ell$. Furthermore, writing $\widehat  p_0 = y$ and $\widehat  p_N = z$, we  have 
$$\|\widehat  p_j - \widehat  p_{j+1}\| \leq  \|\widehat  p_j - p_j\| +  \| p_j -  p_{j+1}\| +  \|p_{j+1} - \widehat  p_{j+1}\| \leq \ve$$ for any $0 \leq j < \ell$. The sequence $s = (\widehat  p_0,\dots,\widehat  p_N)$ is thus a path in $\cG_\ve$ and so
\beq
\wh d_\ve(p,q) \leq L_{s} =  \|\widehat  p_1 - \widehat  p_0\|+\dots+\|\widehat  p_N - \widehat  p_{N-1}\| \leq |\widehat  t_1 - \widehat  t_0|+\dots+|\widehat  t_N - \widehat  t_{N-1}| = \wh t_N - \wh t_0 \nonumber
\eeq
where we set $\widehat  t_0 = 0$ and $\widehat  t_N = \ell = d_M(p,q)$, ending the proof. 
\end{proof}

\begin{proof}[Proof of \lemref{dd2}] Following the proof of \citep[Lem. 5]{AG19} if there exists $\delta > 0$ such that $\|y-z\| \leq \delta$ implies $d_M(y,z) \leq \pi \tau$ for all $y,z \in M$, then we must have that for any $y,z \in M$ satisfying $\|y-z\| \leq \delta$,
\beq
d_M(x,y) \leq \(1 + \frac{\pi^2}{48 \tau^2}\|y-z\|^2\)\|y-z\|. \nonumber
\eeq
Thanks to \prpref{eqdist}, this must hold for $\delta = \tau/2$. Now let $p_0, \dots, p_m$ be the shortest path in $\cG_\ve$ between $y$ and $z$. Since $\ve \leq \tau/2$, we have
\begin{align*}
d_M(p,q) &\leq \sum_{j=1}^m d_M(p_j,p_{j-1}) \leq \sum_{j=1}^m \(1 + \frac{\pi^2}{48 \tau^2}\|p_j-p_{j-1}\|^2\)  \|p_j-p_{j-1}\| \\
&\leq \(1 + \frac{\pi^2}{48 \tau^2} \ve^2\)   \widehat d_\ve(p,q) 
\end{align*}
which ends the proof.
\end{proof}

In view of Lemma \ref{lem:dd1} and Lemma \ref{lem:dd2}, we want to tune $\ve$ so that it is the smallest possible and so that $16 \widehat  \eta \leq \ve$ holds with high probability. This is achieved for $\ve$ of order $\log n/n$.
 
 \begin{lem} \label{lem:eta}
Setting $\ve = \frac{32(p+1) \log n}{f_{\min} n}$, for every $n \geq 3$, we have $\bbP\(16\widehat  \eta \leq \ve\) \geq 1 - 1/n^p$. 
 \end{lem}
 \begin{proof} Let $\delta > 0$, and let $N = \floor{L_M/\delta}$. We split $[0, L_M]$ into $N$ intervals $I_1,\dots,I_N$ of length $L_M / N$. We denote $\cA$ the event for which each $I_j$ contains at least one coordinate among those of the sample of observations $(X_1, \dots, X_n)$. On $\cA$, we have  $\widehat  \eta \leq L_M / N \leq 2\delta$. Moreover,
 \begin{align*}
 \bbP(\cA) &= 1 - \bbP\(\exists j,~\gamma(I_j)~\text{contains no observation}\) \\
 &\geq 1 - N \(1 - \min_{1\leq j \leq N} P(\gamma(I_j))\)^n  \geq 1 - N\(1 - \frac{a L_M}{N}\)^n.
 \end{align*}
 Using that $N \leq L_M/\delta$ and that $L_M \leq 1/a$ we infer
 \beq
 \bbP\(\widehat  \eta \leq 2 \delta\) \geq 1 - \frac{1}{a\delta} (1 - a\delta)^n \geq  1 - \frac{e^{-a\delta n}}{a \delta}. \nonumber
 \eeq
 Setting $\delta = \frac{(p+1) \log n}{a n}$ and $\ve = 32\delta$ yields
 \beq
 \bbP\(16 \widehat  \eta \leq  \ve\) \geq 1 - \frac{n}{(p+1) n^{p+1} \log n} \geq 1 - \frac1{n^p} \nonumber
 \eeq
 as soon as $\log n \geq 1$, {\it i.e.} for $n \geq 3$. 
 \end{proof}

 \begin{proof}[Proof of \prpref{kerd1}] Recall that we set $K^{\od} = K^{(1,\ell)}$ where $K^{(1,\ell)}$ is defined starting from kernel $\lambda_1^{-1} \Lambda$. Let $\cA$ be the event  $\{16 \widehat  \eta \leq \ve\}$. By triangle inequality, $\bbE_{P^{\otimes n}}[|\hat f^{\od}_h(x) - f_P(x)|^p]^{1/p} \leq \cR_\cA + \cR_{\cA^c}$, with
 \beq
 \cR_\cA = \(\bbE_{P^{\otimes n}}\big[| \hat  f^{\od}_h(x) - f_P(x)|^p\ind_{\cA}\big]\)^{1/p}~~~\text{and}~~~\cR_{\cA^c} = \(\bbE_{P^{\otimes n}} \big[| \hat  f^{\od}_h(x) - f_P(x)|^p \ind_{\cA^c}\big]\)^{1/p}. \nonumber
 \eeq
 On $\cA$, we have, for $n$ large enough (depending on $p, f_{\min}$ and $\tau$) such that $\ve \leq \tau/2$ holds, $| \wh d_\ve(X_i,x) - d_M(X_i,x) | \leq C_1 \ve^2$ with $C_1$ depending on $\tau$ only. This is infered by Lemmas \ref{lem:dd1} and \ref{lem:dd2}. We deduce that, on this event, 
 \beq
 \left|  \hat  f_h^{\od}(x) - \hat  g^{\od}_h(x)\right | \leq \frac{C_1 \|{K^{\od}}'\|_\infty \ve^2}{h^2}~~~\text{with}~~~ \hat g^{\od}_h(x) = \frac{1}{n} \sum_{i=1}^n K^{\od}_h( d_M(X_i,x)). \nonumber
 \eeq
 It follows that
 \begin{align*}
  \cR_\cA  & \leq \frac{C_1 \|{K^{\od}}'\|_\infty \ve^2}{h^2} +  \(\bbE_{P^{\otimes n}}\big[| \hat  g^{\od}_h(x) - f_P(x)|^p\big]\)^{1/p} \\
  &\leq \frac{C_1 \|{K^{\od}}'\|_\infty \ve^2}{h^2} +  \(\bbE_{P^{\otimes n}}\big[|\widehat \xi_h^*(P,x)|^p\big]\)^{1/p} + |\cB_h^*(P,x)|
 \end{align*} 
 with $\cB_h^*$ and $\widehat  \xi_h^*$ denoting the bias and stochastic deviation of estimator $\hat  g^{\od}_h(x)$. Following the same arguments as in proof of \prpref{sd}, we have $\bbE_{P^{\otimes n}}[|\widehat \xi_h^*(P,x)|^p]^{1/p}  \leq c_p \Omega(h)^p$ with $c_p$ depending only on $p$. For the bias term, as soon as $h \leq \pi\tau$, we have
 \begin{align*}
\cB^*_h(P,x) &= \bbE_{P^{\otimes n}} [ \hat  g^{\od}_h(x)] - f(x) = \int_M K^{\od}_h(d(p,x)) f(p) d\mu_M(p) -f(x)\\
 &=  \int_{B_{T_x M}(0,1)} K^{\od}(\|v\|) \(f\circ \exp_x(hv) - f(x) \)dv.
 \end{align*}
 Since now $f\circ \exp_x$ is $\beta$-H\"older on $B_{T_x M}(0,\pi\tau)$, we know that all the terms in the development of $\cB_h^*(P,x)$ up to order $\ceil{\beta-1}$ cancels. We deduce $|\cB_h^*(P,x)| \leq C_2 h^\beta$ with $C_2$ depending on $\ell$ and $R$ only. For the other term $\cR_{\cA^c}$, we write $|\hat  f^{\od}_h(x) - f(x) | \leq \frac{\|K^{\od}\|_\infty}{h} + f_{\max},$ so that, according to \lemref{eta},
 \beq
 \cR_{\cA^c} \leq \(\frac{\|K^{\od}\|_\infty}{h} + f_{\max}\) \bbP(\cA^c)^{1/p} \leq C_3 \frac{1}{nh} \nonumber
 \eeq
 with $C_3$ depending on $\ell$ and $f_{\max}$. Putting all these estimates together yields the result.
 \end{proof}

\subsection{Proofs of \secref{smooth}} \label{app:smooth}

\begin{lem} \label{lem:hstar}
For any $P \in \sdab$, and $\Theta > p$, we have
\beq
 \bbE_{P^{\otimes n}}[|\wh f(x) - f_P(x)|^p]^{1/p} \lesssim \Omega(h^*(P,x))\lambda(h^*(P,x)) \nonumber
\eeq
 up to a constant depending on $p$ and $\Theta$, with
\beq
h^*(P,x) = \max\big\{h \in \bbH~|~\forall \eta \in \bbH(h),~ |f_\eta(P,x) - f(x)| \leq \frac12 \Omega(h) \lambda(h)\big\}. \nonumber
\eeq 
\end{lem}

\begin{proof}
We fix $P \in \sdab$ and write $\widehat  h$ and $h^*$ for $\widehat  h(x)$ and $h^*(P,x)$ respectively. Let $\mathcal A = \{\widehat  h\geq h^*\}$. 
We can write $\bbE_{P^{\otimes n}}[|\wh f(x) - f_P(x)|^p]= \cR_\cA + \cR_{\cA^c}$, where
\beq
 \cR_\cA = \bbE_{P^{\otimes n}}[ |\widehat  f(x) - f_P(x)|^p \ind_{\mathcal A}]~~~\text{and}~~~ \cR_{\cA^c} = \bbE_{P^{\otimes n}} [ |\widehat  f(x) - f_P(x)|^p \ind_{\mathcal A^c}]. \nonumber
\eeq
We start with bounding $\cR_\cA$. Firstly,
\beq
\begin{split}
\cR_\cA \leq 3^{p-1}\big(&\bbE_{P^{\otimes n}} [ |\widehat  f_{\widehat  h}(x) - \widehat  f_{h^*}(x)|^p \ind_{\mathcal A}] + \bbE_{P^{\otimes n}}[ |\widehat  f_{h^*}(x) - f_{h^*}(P,x)|^p \ind_{\mathcal A}] \\&+ \bbE_{P^{\otimes n}}[ | f_{h^*}(P,x) - f_P(x)|^p \ind_{\mathcal A}]\big). \nonumber
\end{split}
\eeq
Next, by definition of $\widehat  h$ and $\mathcal A$, we have  
$$|\widehat  f_{\widehat  h}(x) - \widehat  f_{h^*}(x)| \ind_{\mathcal A} \leq \psi(\widehat  h,h^*)\ind_{\mathcal A} \leq  2\Omega(h^*)\lambda(h^*).$$
By definition of $h^*$, we also have $| f_{h^*}(P,x) - f(x)| \leq \frac12 \Omega(h^*) \lambda(h^*)$. Finally, using 
Proposition \ref{prp:sd}
$$\bbE_{P^{\otimes n}}[ |\widehat  f_{h^*}(x) - f_{h^*}(P,x)|^p \ind_{\mathcal A}] \leq c_p \Omega(h^*)^p \leq c_p (\Omega(h^*) \lambda(h^*))^p$$
holds as well. Putting all three inequalities together yields
\beq
\cR_{\cA} \leq C_{\mathcal A} (\Omega(h^*) \lambda(h^*))^p~~~\text{with}~~~C_{\mathcal A} = 3^{p-1}\(2^p + c_p + 2^{-p}\). \nonumber
\eeq
We now turn to $ \cR_{\cA^c}$. Notice that for any $h \in \bbH(h^*)$, we have
\beq
|f_h(P,x) - f(x)| \leq \frac12 \Omega(h^*)\lambda(h^*) \leq  \frac12 \Omega(h)\lambda(h),  \nonumber
\eeq
hence
\beq
|\widehat  f_h(x) - f(x)| \leq  \frac12 \Omega(h^*)\lambda(h^*) + |\widehat  \xi_h(P,x)|. \nonumber
\eeq
We can thus write
\beq
 \cR_{\cA^c} = \sum_{h \in \bbH(h^*/2)}    \bbE_{P^{\otimes n}} [|\widehat  f_h(x) - f(x)|^p \ind_{\{\widehat  h = h\}}] \leq \sum_{h \in  \bbH(h^*/2)} \bbE_{P^{\otimes n}}\big[\big(\frac12 \Omega(h^*)\lambda(h^*) + |\widehat  \xi_h(P,x)|]\big)^p \ind_{\{\widehat  h = h\}} \big]. \nonumber
\eeq
Now, for any $h \in \bbH(h^*/2)$, we have
\begin{align*}
\{\widehat  h = h\} &\subset \{\exists \eta \in \bbH(h),~ |\widehat  f_{2h}(x) - \widehat  f_\eta(x)| > \psi(2h,\eta)\} \\
&\subset \bigcup_{\eta \in \bbH(h)} \left\{ \Omega(h^*) \lambda(h^*) + |\widehat  \xi_{2h,\eta}(P,x)| > \psi(2h,\eta)\right\},
\end{align*}
where $\widehat  \xi_{2h,\eta}(P,x) = \widehat  \xi_{2h}(P,x) - \widehat  \xi_{\eta}(P,x)$, and where we used the triangle inequality and the definition of $h^*$. Now, we have $\Omega(h^*) \lambda(h^*) \leq \Omega(2h) \lambda(2h)$ since $2h \leq h^*$ and by definition of $\psi(2h,\eta)$, we infer 
\beq
\{\widehat  h = h\} \subset \bigcup_{\eta \in \cH(h)} \left\{ |\widehat  \xi_{2h,\eta}(P,x)| >  \Omega(\eta)\lambda(\eta)\right\} \nonumber
\eeq
so that
\begin{align}
\bbP(\widehat h = h) &\leq \sum_{\eta \in \bbH(h)} \bbP\( |\widehat  \xi_{2h,\eta}(P,x)| > \Omega(\eta)\lambda(\eta)\)  \label{dl} \\
&\leq  \sum_{\eta \in \bbH(h)} \bbP\( |\widehat  \xi_{2h,\eta}(P,x)| > \sqrt{\frac{8\omega  \lambda(\eta)}{n\eta^d}} + \frac{2\|K\|_\infty \lambda(\eta)}{n\eta^d} \) \nonumber \\
&\leq  \sum_{\eta \in \bbH(h)} 2 \exp(-  \lambda(\eta)^2). \label{ber}
\end{align}
For \eqref{dl} we use the fact that $ \lambda(\eta) \geq 1$ and Bernstein's inequality on the random variable $\widehat \xi_{2h,\eta}(P,x)$ for \eqref{ber}. Noticing now that $\lambda(\eta)^2 \geq d \Theta \log(1/\eta)$, we further obtain
\begin{align*}
\bbP(\widehat h = h) &\leq  2 h^{ \Theta  d} \times \sum_{j= 0}^{\floor{\log_2(1/h^-)}}  2^{- j \Theta d} \leq  \frac{2}{1 - 2^{-\Theta d}} h^{\Theta d}.
\end{align*}
For any $h \in \bbH(h^*/2)$, we thus get the following bound, using Cauchy-Schwarz inequality
\begin{align}
\mathbb E_{P^{\otimes n}}\big[\big(\frac12 \Omega(h^*)\lambda(h^*) &+ |\widehat  \xi_h(P,x)|]\big)^p \ind_{\{\widehat  h = h\}} \big]   \nonumber\\
&\leq  \bbP( \widehat h = h)^{1/2} E_{P^{\otimes n}}\big[\big(\frac12 \Omega(h^*)\lambda(h^*) + |\widehat  \xi_h(P,x)|]\big)^{2p}\big]^{1/2} \nonumber\\
&\leq 2^{(2p-1)/2} \sqrt{\frac{2}{1 - 2^{-\Theta d}}} h^{\Theta d/2} \(2^{-p} \Omega(h^*)^p \lambda(h^*)^p + c_{2p}^{1/2} \Omega(h)^p \). \label{cs}
\end{align}
We plan to sum over $h \in \bbH(h^*/2)$ the RHS of \eqref{cs}. Notice first that 
$$\sum_{h < h^*} h^{\Theta d / 2} \leq (h^*)^{\Theta d / 2} (1 - 2^{- \Theta d/ 2})^{-1}.$$
Moreover, for any $h \geq h^-$, we have $\Omega(h) \leq 2 \sqrt{2\omega/(nh^d)}$ by definition of $h^-$. It follows that 
\beq
\Omega(h^*) \leq \Omega(h) \leq 2\Omega(h^*) \(\frac{h^*}{h}\)^{d/2}. \nonumber
\eeq
for any $h \leq h^*$. This enables us to bound the following sum
\begin{align*}
\sum_{h \in \bbH(h^*/2)} h^{\Theta d / 2} \Omega(h)^p &\leq 2^p \Omega(h^*)^p {h^*}^{pd/2} \sum_{h \in \bbH(h^*/2)} h^{\Theta d / 2 - pd/2}  \\
&\leq \frac{2^p}{1 - 2^{(p - \Theta) d/ 2}} \Omega(h^*) {h^*}^{\Theta d/2}
\end{align*}
where we used that $\Theta > p$. Putting all these estimates together, using that $h^* \leq 1$ and $\lambda(h^*) \geq 1$, we eventually obtain
\beq
\cR_{\cA^c} \leq C_{\cA^c}  \Omega(h^*)^p \lambda(h^*)^p ~~ \text{with} ~~ C_{\cA^c} = \frac{2^p}{ \sqrt{1 - 2^{-\Theta d}}} \( \frac{2^{-p}}{1 - 2^{-\Theta d/2}} + \frac{\sqrt{c_{2p}} 2^p}{1 - 2^{(p - \Theta)d/2}}\).  \nonumber
\eeq
In conclusion $ \bbE_{P^{\otimes n}}[|\wh f(x) - f_P(x)|^p]^{1/p}  \leq (C_\cA + C_{\cA^c}) \Omega(h^*)^p \lambda(h^*)^p$ which completes the proof.
\end{proof}

\begin{proof}[Proof of \thmref{adapt}.] 
Let $P \in \sdab$ and let $\bar h = (\rho \log n /n)^{1/(2\gamma + d)}$ with $\gamma = \alpha \wedge \beta$ and for some constant $\rho$ to be specified later. By \prpref{bias} we know that for $n$ large enough (depending on $\rho, \alpha, \beta, d$) such that $\bar h \leq  \tau/2$, we have $|f_{\eta}(P,x) - f(x)| \leq C_1 \eta^{\gamma}$ for all $\eta \leq \bar h$ with $C_1$ depending on $K, \ell, \alpha, \tau, L, \beta, f_{\max}$ and $R$. Moreover, we also have
\beq
\frac{2^{-2} \Omega(\bar h)^2 \lambda(\bar h)^2}{C_1^2 \bar h^{2\gamma}} \geq \frac{ d \Theta 2\omega \log(1/\bar h)}{4 C_1^2 n \bar h^{2\gamma+ d}} = \frac{d \Theta \omega (2\gamma+d)^{-1}}{2 C_1^2 \rho} \frac{\log n - \log \log n - \log \rho }{\log n}. \nonumber
\eeq
Thus, picking $\rho =   d \Theta \omega (2\gamma+d)^{-1} / (2C_1^2)$ yields $C_1 \bar h^{\gamma} \leq \frac12 \Omega(\bar h) \lambda(\bar h)$ for $n$ large enough (depending on $\rho$), and therefore $\bar h \leq h^*(P,x)$. By \lemref{hstar} this implies 
\beq
 \bbE_{P^{\otimes n}}[|\wh f(x) - f_P(x)|^p]^{1/p}  \leq C_2 \Omega(\bar h) \lambda(\bar h) \nonumber
\eeq
where $C_2$ depends on $p$ and $\Theta$. But using that both $\bar h \geq h^-$ and $\lambda(\bar h)^2 = d\Theta \log(1/\bar h)$ for $n$ large enough (depending on $\rho, d, K$ and $\Theta$), we also obtain
$$
\Omega(\bar h)^2 \lambda(\bar h)^2 \leq \frac{8\omega d\Theta\log(1/\bar h)}{n\bar h^d} = \frac{8\omega d\Theta(2\gamma+d)^{-1}}{\rho} \frac{\log n - \log \log n - \log \omega}{\log n} \bar h^{2\gamma} 
\leq  16 C_1^2 \bar h^{2\gamma}.
$$
This last estimate yields
\beq
\bbE_{P^{\otimes n}}[|\wh f(x) - f_P(x)|^p]^{1/p}  \leq \(4 C_1 C_2 \rho^{\gamma/(2\gamma+ d)} \) \(\frac{\log n}{n} \)^{\gamma/(2\gamma + d)} \nonumber
\eeq
for $n$ large enough depending on $\rho, \alpha, \beta ,d, K$ and $\Theta$, which completes the proof. 
\end{proof}

\subsection{Proofs of \secref{dim}} \label{app:dim}

\begin{proof}[Proof of \prpref{dimadapt}] By the triangle inequality, for any $P \in \sdab$, we write
\beq
\bbE_{P^{\otimes n}}\big[|\wh f_{\widehat h}(\wh d; x) - f_P(x)|^p\big]^{1/p} \leq \Big(\bbE_{P^{\otimes n}}\left[\left|\wh f_{\widehat h}(\wh d; x) - f(x)\right|^p \ind_{\{\widehat d = d\}}\right]\Big)^{1/p} +  \Big(\bbE_{P^{\otimes n}}\left[\left|\wh f_{\widehat h}(\wh x; \cdot) - f(x)\right|^p \ind_{\{\widehat d \neq d\}}\right]\Big)^{1/p} . \nonumber
\eeq
The first term in the right-hand side has the right order thanks to \thmref{adapt}. For the second one, using that $|f(x)| \leq f_{\max}$ and   
\beq
|f_{\widehat h}(\wh d, x)| \leq \sup_{1\leq d < D} \frac{\|K^{(\ell)}(d,\cdot)\|_\infty}{(h_d^-)^d} \lesssim n \nonumber
\eeq
up to a constant that depend on $D, K$ and $\ell$, we infer
\beq
\Big( \bbE_P\left[\left|\wh f_{\widehat h}(\wh d; x) - f(x)\right|^p \ind_{\widehat d \neq d}\right]\Big)^{1/p} \lesssim \bbP\(\widehat d \neq d\)^{1/p} \times n. \nonumber
\eeq
Finally, since $\wh d$ satisfies \assref{dim}, we have 
\beq
\bbE_{P^{\otimes n}}\big[|\wh f_{\widehat h}(\wh d; x) - f_P(x)|^p\big]^{1/p}  \lesssim \(\frac{\log n}{n}\)^{\alpha\wedge \beta/(2\alpha\wedge\beta + d)} + n^{-1/2} \nonumber
\eeq
for $n$ large enough depending on $p,\Theta, K, \ell, \alpha, \tau, L, \beta, f_{\max}, f_{\min}$ and $R$, so that the result indeed holds up to a constant depending on the same parameters and $D$.
\end{proof}

\begin{proof}[Proof of \prpref{dimest}] Let $P \in \sdab$ and $\eta > 0$. Assume that $\wh P_\eta > 0$. We have 
\begin{align*}
|\wh \delta_\eta - d| &\leq |\log_2 \wh P_{2\eta} - \log_2 P_{2\eta}| + |\log_2 \wh P_{\eta} - \log_2 P_{\eta}| + |\log_2 P_{2\eta} - \log_2 P_{\eta} - d | \\
&\leq \frac{1}{\log 2}\(\frac{|\wh P_{2\eta} - P_{2\eta}|}{\wh P_{2\eta} \wedge P_{2\eta}} +  \frac{|\wh P_{\eta} - P_{\eta}|}{\wh P_{\eta} \wedge P_{\eta}}\) + \left | \log_2 \(P_{2\eta} / (2^d P_\eta)\)\right|
\end{align*}
We first consider the determinist term. For $\eta \leq \tau/2$, we have, writing $r_\eta = \xi(\eta/\tau)\eta$ and using \lemref{abstand},
$$L_{2\eta} (1 - \eta^2/6\tau^2) (2\eta)^2 \zeta_d \leq P_{2\eta} \leq U_{2\eta} (1 + r_{2\eta}^2/\tau^2) r_{2\eta}^d \zeta_d$$
and
$$
L_\eta (1 - \eta^2/6\tau^2)\eta^d\zeta_d \leq P_\eta \leq U_{\eta} (1 + r_{\eta}^2/\tau^2) r_{\eta}^d \zeta_d,
$$
where $L_\eta = \inf_{M\cap B(x,\eta)} f$ and $U_{\eta} = \sup_{M \cap B(x,\eta)} f$. Using again \lemref{abstand}, we have that for $\eta \leq \tau/2$, $M \cap B(x,\eta)\subset \exp_x B_{T_x M}(0,2\eta)$, and, since $2\eta \leq \pi\tau/2$ and that $f\in \cF_\beta$, we know, using \lemref{tech}, that there exists $R_1 > 0$ (depending on $\beta,\tau, f_{\max}$ and $R$) such that $f(x) - R_1 \eta \leq L_\eta \leq U_\eta \leq f(x) + R_1 \eta$. If $\eta < \tau/4$, the same bounds apply for $2\eta$ and we thus obtain
$$(f(x) - R_1 2\eta) (1 - \eta^2/6\tau^2)(2\eta)^d \zeta_d \leq P_{2\eta} \leq (f(x) + R_1 2\eta) (1 + r_{2\eta}^2/\tau^2) r_{2\eta}^d \zeta_d$$
and
\begin{equation} \label{peta}
(f(x) - R_1 \eta) (1 - \eta^2/6\tau^2) \eta^d \zeta_d \leq P_\eta \leq (f(x) +R_1 \eta)(1 + r_{\eta}^2/\tau^2) r_{\eta}^d \zeta_d. 
\end{equation} 
Using these two inequalities, and the fact that $r_\eta /\eta \rightarrow 1$ as $\eta \rightarrow 0$, we find that $|P_{2\eta}/(2^d P_\eta) - 1| \lesssim \eta$ up to a constant that depends on $R_1, \tau$ and $f_{\min}$, for $\eta$ small enough (depending on $R_1, \tau$ and $f_{\min}$ as well). For the other terms, a simple use of Hoeffding's inequality yields for any $\eta,\ve > 0$,
\beq
\bbP\(|\wh P_\eta - P_\eta| > \ve \) \leq 2 \exp(-2n\ve^2). \nonumber
\eeq
On the event $\cA_\eta = \{|\wh P_\eta - P_\eta| \leq \ve \}$, we have moreover $\wh P_\eta \wedge P_\eta \geq P_\eta -\ve$. Setting $\ve = \eta^{d+1}$, and using \eqref{peta}, we see that $P_\eta -\ve \gtrsim \eta^d$ for $\eta$ small enough (depending on $R_1,\tau$ and $f_{\min}$). Thus, on the event $\cA_\eta \cap \cA_{2\eta}$, with probability at least $1 - 4\exp(-2n\eta^{2d+2})$, we derive
\beq |\widehat \delta_\eta - d | \lesssim  \eta + \frac{\ve}{P_\eta -\ve } + \frac{\ve}{P_{2\eta} -\ve}  \lesssim \eta,  \label{eta}
\eeq
for $\eta$ small enough (depending on $R_1,\tau$ and $f_{\min}$),  up to a constant that depends on  $R_1,\tau$ and $f_{\min}$.
Now setting $\eta = n^{-1/(2D+2)}$, we have $\wh d = \wh \delta_\eta = d$ on the event $\cA_\eta \cap \cA_{2\eta}$ as soon as $n$ is large enough so that the LHS of \eqref{eta} is strictly smaller than $1/2$, ending the proof.
\end{proof}


\end{document}